\documentclass[a4paper,11pt, twoside]{article}
\usepackage{a4} 
\usepackage{geometry} 
\usepackage{amsmath}
\usepackage{amssymb}
\usepackage{amsthm}
\usepackage{graphicx}
\usepackage{caption,subcaption}
\usepackage{multirow}
\usepackage{xstring}
\usepackage[utf8]{inputenc}
\usepackage{amstext,amsbsy,amsopn,eucal,enumerate}
\usepackage{tikz}
\usepackage{pgfplots}
\usepackage{tabularx}
\def\R{\mathbb{R}}

\newcommand{\expn}{\operatorname{e}}

\newcommand{\im}{\operatorname{im}}
\newcommand{\diag}{\operatorname{diag}}

\newcommand{\beq}{\begin{equation}}
\newcommand{\eeq}{\end{equation}}
\newcommand {\mat}      [1] {\left[\begin{array}{#1}}
\newcommand {\rix}          {\end{array}\right]}
\newcommand {\smat}      [1] {\left[\begin{smallmatrix}{#1}}
\newcommand {\srix}          {\end{smallmatrix}\right]}
\newcommand {\s}      [1] {\begin{smallmatrix}{#1}}
\newcommand {\se}          {\end{smallmatrix}}
\newcommand{\trace}{\operatorname{tr}}

\newtheorem{defn}{Definition}[section]
\newtheorem{remark}[defn]{Remark}
\newtheorem{example}[defn]{Example}
\newtheorem{lem}[defn]{Lemma}
\newtheorem{prop}[defn]{Proposition} 
\newtheorem{kor}[defn]{Corollary}
\newtheorem{thm}[defn]{Theorem}

\title{Model reduction for stochastic systems with nonlinear drift}

\author{Martin Redmann\thanks{Martin Luther University Halle-Wittenberg, Institute of Mathematics, Theodor-Lieser-Str. 5, 06120 Halle (Saale), Germany, Email: {\tt 
martin.redmann@mathematik.uni-halle.de}.}
}

\begin{document}

\maketitle

\begin{abstract}
In this paper, we study dimension reduction techniques for large-scale controlled stochastic differential equations (SDEs). The drift of the considered SDEs contains a polynomial term satisfying a one-sided growth condition. Such nonlinearities in high dimensional settings occur, e.g., when stochastic reaction diffusion equations are discretized in space. We provide a brief discussion around existence, uniqueness and stability of solutions. (Almost) stability then is the basis for new concepts of Gramians that we introduce and study in this work. With the help of these Gramians, dominant subspace are identified leading to a balancing related highly accurate reduced order SDE. We provide an algebraic error criterion and an error analysis of the propose model reduction schemes. The paper is concluded by applying our method to spatially discretized reaction diffusion equations.
\end{abstract}



\pagestyle{myheadings}
\thispagestyle{plain}
\markboth{M. REDMANN}{MODEL REDUCTION FOR STOCHASTIC SYSTEMS WITH NONLINEAR DRIFT}


\section{Introduction}

Model order reduction (MOR) aims to find low-order approximations for high-/infinite-dimensional systems of differential equations reducing the complexity of the original problem. Many MOR schemes are based on projections (Galerkin or Petrov-Galerkin type). In this context, the first goal is to identify solution manifolds and approximate them by low-dimensional linear subspaces. A reduced state variable, taking values in this subspace, is subsequently constructed in order to ensure an accurate estimation of the original dynamics. There is a rich selection of different MOR strategies. Proper orthogonal decomposition (POD) \cite{pod} is an approach, where solution spaces are learned from data. Methods like the iterative rational Krylov algorithm (IRKA) \cite{irka} rely on interpolation or on the minimization of certain error measures between systems. Moreover, there are Gramian based techniques like balanced truncation (BT) \cite{moo1981}, where dominant subspaces of the original dynamics are associated to eigenspaces of these (algebraic) Gramians. Recently, there has been an enormous interest in dimension reduction for large-scale nonlinear systems. Data-driven \cite{gosea_antoulas, pod_kramer, qian_data} or interpolation/optimization based methods \cite{breiten_benner2, nonlinear_irka} were applied to such equations in a deterministic framework. Generalizing BT to nonlinear systems was first addressed in \cite{Scherpen}. Alternatives, where the reduced order model can be computed easier, can be found in \cite{morBenG17, Kramer2022}. \smallskip

MOR in probabilistic settings is even more essential than in the deterministic context discussed above. This is due to an enormous amount of system evaluations required, e.g., for conducting Monte-Carlo simulations. On the other hand, it is also about the feasibility of certain algorithms. E.g., a stochastic differential equation (SDE) in dimension $n$ is in some sense equivalent to a partial differential equation (PDE) with $n$ spatial variables using the formula of Feynman-Kac. Knowing how hard it is to solve high-dimensional PDEs in general, it becomes clear how vital MOR for SDEs is. A POD approach for SDEs is studied in \cite{pod_sde}. Balancing related or optimization based MOR techniques are, for instance, investigated in \cite{beckerhartmann,bennerdammcruz, redmannbenner, mliopt} for the linear case. The advantage of the latter schemes is the possibility for a detailed error and stability analysis. However, an extension to nonlinear stochastic systems seems very challenging. A first approach for stochastic bilinear equations is presented in \cite{redstochbil} but it might not work for more complex nonlinearities.\smallskip

The goal of this paper is to extend BT to stochastic systems, e.g., with certain polynomial nonlinearities. In the deterministic case, a wide focus is on quadratic systems, see for instance \cite{morBenG17, Kramer2022}. This is because many nonlinear terms in a differential equation can be transformed to a quadratic expression using additional dummy variables. This approach is called lifting in the literature. It has the advantage that a large set of nonlinear systems can be covered if we know how to handle quadratic ones. However, this is also the drawback of this ansatz, since differential equations involving quadratic terms range from globally stable to finite time explosion systems, i.e., the existence of a global solution is not guaranteed. This large variety of properties makes it seem infeasible to develop a general theory like for example an error analysis with sharp bounds. For that reason, we do not intend to apply the technique of lifting the dynamics to a quadratic system in this paper, because one might loose track of essential properties that are usually not visible anymore in a transformed SDE. Instead we exploit the structure of our locally Lipschitz nonlinearity that we assume to be of one-sided linear growth. This also involves interesting polynomials that play a role in reaction diffusion equations. This type of growth will be reflected linearly in the associated Lyapunov operator that defines the Gramians that we propose in our MOR procedure.\smallskip

The paper is now structured as follows. Section \ref{stochstabgen} deals with the setting and the first details concerning the goals of this work. In Section \ref{sec3}, we recall facts about existence and uniqueness of solutions to the considered nonlinear SDE. We further investigate global asymptotic stability as the basis of the Gramians that we introduce in Section \ref{sec:reach}. There, it is explained and reasoned how Gramians need to be chosen in order to find a good dominant subspace characterization and hence an accurate reduced system. We also discuss on properties of Gramians that need to be fulfilled to ensure the classical error bound for BT known for deterministic linear systems \cite{BT_bound_enns, BT_bound_glover}. Having computed the desired Gramians based on the strategy that we provide, we explain how to compute the reduced system in Section \ref{sec_BT}. Finally, Section \ref{sec_error_analysis} delivers an error bound analysis for the balancing related MOR scheme, also involving a discussion on criteria for a high approximation quality. Section \ref{sec_sim} illustrates the performance of the MOR technique by applying it to spatially discretized stochastic reaction diffusion equations.

\section{Setting, notation and goal}\label{stochstabgen}

Let $\left(\Omega, \mathcal F, (\mathcal F_t)_{t\in [0, T]}, \mathbb P\right)$\footnote{$(\mathcal F_t)_{t\in [0, T]}$ is right continuous and complete.} be a filtered probability space on which every stochastic process appearing in this paper is defined. Given an $\mathbb R^d$-valued and square integrable L\'evy process $M=\begin{bmatrix}M_1 & \ldots & M_d\end{bmatrix}^\top$ with mean zero, we assume that it is $(\mathcal F_t)_{t\in [0, T]}$-adapted and its 
increments $M(t+h)-M(t)$ are independent of $\mathcal F_t$ for $t, h\geq 0$ and $t+h\leq T$. Exploiting the independent and stationary increments, there exists a positive semidefinite matrix $K=\left(k_{ij}\right)_{i, j=1, \ldots, d}$, so that $\mathbb E[M(t)M(t)^\top]=K t$, see \cite[Theorem 4.44]{zabczyk} for a proof. We call $K$ covariance matrix of $M$.
Now, we consider the following large-scale nonlinear stochastic dynamics driven by $M$:
 \begin{subequations}\label{original_system}
\begin{align}\label{stochstatenew}
             dx(t)&=[Ax(t)+Bu(t)+ f\left(x(t)\right)]dt+N\left(x(t-)\right)dM(t),\quad x(0)=x_0,\\ \label{output_eq}
            y(t) &= Cx(t),\quad t\in [0, T],
\end{align}
\end{subequations}
where $x(t-):=\lim_{s\uparrow t} x(s)$, $A\in \mathbb R^{n\times n}$, $B\in \mathbb R^{n\times m}$, $C\in \mathbb R^{p\times n}$, $N: \mathbb R^n\rightarrow \mathbb R^{n\times 
d}$ is a linear mapping
defined by $N(x)=\begin{bmatrix}N_1 x &\ldots & N_d x\end{bmatrix}$ for $x\in\mathbb R^n$ with $N_1, \ldots, N_d\in\mathbb R^{n\times n}$. The state vector $x(t)\in\mathbb R^n$ is assumed to be high-dimensional, whereas the quantity of interest $y(t)\in\mathbb R^p$ usually is a vector with a low number of entries.  The nonlinear function $f:\mathbb R^n\rightarrow \mathbb R^{n}$ shall satisfy the following local Lipschitz condition
\begin{align}\label{loc_lip__cond}
  \left\|f(x)-f(z) \right\|_2\leq c_R  \left\|x-z\right\|_2,                                                                                                                                                                                                                                                                                                                                                                                                                                                                                                                                                                                                                                                                                                                          \end{align}
for $\left\|x\right\|_2, \left\|z\right\|_2\leq R$, $c_R>0$  and any $R>0$, where $\langle \cdot, \cdot \rangle_2$ denotes the Euclidean inner product with corresponding norm $\left\|\cdot\right\|_2$. Further, we assume the special type of monotonicity condition\begin{align}\label{monotonicity_cond}
  \langle x, f(x) \rangle_2 \leq c_f   \left\|x\right\|_2^2,                                                                                                                                                                                                                                                                                                                                                                                                                                                                                                                                                                                                                                                                                                                          \end{align}
for all $x\in \mathbb R^n$ and a constant $c_f$. In the literature, \eqref{monotonicity_cond} is called one-sided growth condition as well. In fact, $c_f$ can be negative. In this case,  \eqref{monotonicity_cond} is also known as dissipativity condition. 
Below, $x(t, x_0, B)$, $t\in [0, T]$,
represents the solution to \eqref{stochstatenew} with initial condition $x_0\in\mathbb R^n$ and matrix $B$ determining the inhomogeneous part of the state equation. The associated control process $u$ is assumed to be an $(\mathcal F_t)_{t\in [0, T]}$-adapted process with 
     \begin{align*}
\left\|u\right\|_{L^2_T}^2:=\mathbb E \int_0^T \left\|u(t)\right\|_2^2 dt<\infty.
\end{align*}            
Moreover, suppose that $f(0) = 0$ to ensure that the uncontrolled state equation \eqref{stochstatenew} ($B= 0$) has an equilibrium at zero. If $f(0)\neq 0$, we can replace $f$ by $f-f(0)$ as well as $B$ and $u$ by $\begin{bmatrix}B & f(0) \end{bmatrix}$ and $\begin{bmatrix}u & 1 \end{bmatrix}^\top$, respectively. The above setting covers interesting polynomial nonlinearities. This fact is illustrated in the next example.
\begin{example}\label{example1}
The local Lipschitz condition \eqref{loc_lip__cond} is fulfilled by all functions $f$ with continuous partial derivatives. This is particularly given for polynomials. If we assume $f = f^{(i)}$, $i\in\{1, 2, 3\}$, to be special third order polynomial, where \begin{align*}
 f^{(1)}(x) &=  x\circ(\mathbf 1_n -x)\circ(x-\mathbf 1_n a)= (1+a) x^{\circ 2} - x^{\circ 3}-a x,\quad  a\in \mathbb R,\\
 f^{(2)}(x) &= x - x^{\circ 3} \quad \text{and} \quad f^{(3)}(x) = x - x \left\|x\right\|_2^2,
 \end{align*}
 the monotonicity condition \eqref{monotonicity_cond} holds.  
The products/powers involving ``$\circ$'' have to be understood in the Hadamard (component wise) sense and $\mathbf 1_n$ is the vector of ones having length $n$. Now,  \eqref{monotonicity_cond} can be verified by the following calculations \begin{align*}
\langle x,  f^{(1)}(x)\rangle_2 &=  -a\left\|x\right\|_2^2 +\sum_{i=1}^n x_i^2[(1+a)x_i-x_i^2] \leq  \frac{(a-1)^2}{4} \left\|x\right\|_2^2, \\                                                                                                                                                                                                       \langle x, f^{(2)}(x) \rangle_2 &= \left\|x\right\|_2^2 -\sum_{i=1}^n x_i^4 \leq \left\|x\right\|_2^2,\quad \langle x, f^{(3)}(x) \rangle_2 = \left\|x\right\|_2^2 - \left\|x\right\|_2^4 \leq \left\|x\right\|_2^2                                                                                                                                                                            \end{align*}
exploiting that $(1+a)x_i-x_i^2\leq \frac{(a+1)^2}{4}$ for all $x_i\in \mathbb R$.
\end{example}

Our setting is not restricted to the functions of Example \ref{example1}. However, we will frequently refer to these interesting cases.  Let us point out that the component-wise functions $f^{(1)}$ and $f^{(2)}$ occur if the nonlinear part of certain (stochastic) reaction diffusion equations are evaluated on a spatial grid. To be more precise, a finite difference discretization of Zeldovich-Frank-Kamenetsky (or FitzHugh-Nagano) and Chafee-Infante equations would lead to such a setting. This paper does not intend to discuss finite difference schemes for stochastic partial differential equations in detail. However, the interested reader may find more information regarding these methods in \cite{gyongy1, gyongy2,gyongy0, shardlow}. We also refer to, e.g., \cite{react_dif_daprato, kuehn_neamtu, marinelli_roeckner, zabczyk} for a theoretical treatment of stochastic reaction diffusion equations.\medskip

The goal of this paper is to drastically reduce the dimension of the high-dimensional system \eqref{original_system} in order to lower the computational complexity when solving this system of stochastic differential equations. Therefore, the solution manifold of \eqref{stochstatenew} shall be approximated by an $r$-dimensional subspace $\im[V]$ of $\R^n$  ($V\in\mathbb R^{n\times r}$ is a full-rank matrix), so that we find a process $x_r$ yielding $V x_r(t) \approx x(t)$. Inserting this estimate into \eqref{original_system} leads to \begin{align}\label{residual_eq}
 V x_r(t) = x_0 + \int_0^t AV x_r(s) + Bu(s) +f(V x_r(s)) ds + \int_0^t {N}\left(V x_r(s-)\right) dM(s) + e(t)
\end{align}
with $y(t)\approx y_r(t):= C V x_r(t)$ and where $e(t)$ is the  state equation error. Now, we enforce the residual $e(t)$  to be orthogonal to a second subspace $\im[W]$ ($W\in\mathbb R^{n\times r}$ has full rank). We further assume that our choice of $W$ provides $W^\top V = I$. Multiplying \eqref{residual_eq} with $W^\top$, we obtain
\begin{subequations}\label{red_system}
\begin{align}\label{red_stochstate}
 dx_r(t) &= [A_r x_r(t)+B_r u(t)+ f_r(x_r(t))] dt + N_{r}(x_r(t-)) dM(t), \\ \label{red_stochout}
 y_r(t) &= C_r x_r(t),\quad t\in [0, T],
\end{align}
\end{subequations}
with  $x_r(0) = W^\top{x}_{0}\in\R^r$, $r\ll n$ and $y\approx y_r$. Generally, we have that $x_r(t)\in\mathbb R^{r}$, $A_r\in \R^{r\times r}$, $B_r\in \mathbb R^{r\times m}$, $C_r\in \R^{p\times r}$, $N_r: \mathbb R^r\rightarrow \mathbb R^{r\times 
d}$ defined by $N_r(x_r)=\begin{bmatrix}N_{r, 1} x_r &\ldots & N_{r, d} x_r\end{bmatrix}$ for $x_r\in\mathbb R^r$, where ${N}_{r, i} \in \R^{r\times r}$ ($i=1,\ldots, d$) and $f_r:\R^r\rightarrow \R^r$. In particular, the reduced coefficients are of the following form 
 \begin{align}\label{red_mat_projection}
   A_r = W^\top A V, \quad    B_{r} = W^\top B,\quad f_r(\cdot) = W^\top f(V \cdot),\quad N_{r, i} = W^\top N_i V,\quad C_r = C V.
     \end{align}
The goal of this paper is to provide a reduced order method for which we can compute the projection matrices $V$ and $W$ and for which we find an accurate approximation of \eqref{original_system}. Here, the main focus will be on the control dynamics and not on the initial state. Therefore, we study reduced order modelling when $x_0=0$. Moreover, we aim to investigate Gramian based schemes which often heavily rely on stability of the state equation. Therefore, we  discuss global asymptotic stability in the next section. Before doing so, we briefly point out that there is a unique solution to \eqref{stochstatenew} by referring to the existing literature.

\section{Existence and uniqueness as well as global asymptotic stability} \label{sec3}

\subsection{Existence and uniqueness for \eqref{stochstatenew}}

We briefly discuss that our setting is well-posed. We define the drift function $F(t, x):= Ax+Bu(t)+ f(x)$ of \eqref{stochstatenew}. Using \eqref{monotonicity_cond} and exploiting that the remaining parts in the drift and diffusion are either linear in $x$ or solely time dependent,  we can find a constant $c_{F, N}$, so that 
\begin{align}\label{loc_lip__cond_F}
 2\langle x, F(t, x) \rangle_2 + \|N(x)K^{\frac{1}{2}}\|^2_F \leq c_{F, N}  \left(1+ \|x\|_2^2\right)
\end{align}
given that the control $u$ is bounded by a constant independent of $t\in[0, T]$ and $\omega\in \Omega$. Here, $\|\cdot\|_F$ denotes the Frobenius norm. Moreover, the drift $F$ is locally Lipschitz continuous (uniformly in $(t, \omega)$) in the sense of \eqref{loc_lip__cond}, since the same is true for $f$. As $N$ is linear, it is particularly globally Lipschitz with respect to $\|\cdot K^{\frac{1}{2}}\|_F$. The monotonicity condition \eqref{loc_lip__cond_F} and local Lipschitz continuity of drift and diffusion yield existence and uniqueness of a solution to \eqref{stochstatenew} by \cite[Theorem 3.5]{mao} if $M$ is a Brownian motion. On the other hand, the arguments of Mao \cite{mao} can immediately be transferred to our more general setting because the Ito-integral w.r.t $M$ has essentially the same properties as the one in the Brownian case. The first property is the Ito isometry $\mathbb E \left\|\int_0^T   \Psi(s) dM(s) \right\|_2^2 =\mathbb E \int_0^T   \|\Psi(s)K^{\frac{1}{2}} \|_F^2 ds=:\|\Psi \|^2$ for predictable\footnote{Predictable means that the process is measurable w.r.t. the $\sigma$ algebra that is generated by left-continuous and $(\mathcal F_t)_{t\in[0, T]}$-adapted processes.} processes $\Psi$ with $\|\Psi \|<\infty$ which relies on the linear covariance function of $M$, see \cite{zabczyk}. Secondly, the equation for the expected value of a quadratic form of the state variable has the same structure, see  Lemma \ref{lemstochdiff}. It is also worth mentioning that existence and uniqueness has been established in a more general setting than in \cite{mao} also covering ours, see \cite{exisence_uniqueness_levy}. There, the result was proved assuming a monotonicity condition, a local Lipschitz condition in the drift and the Brownian diffusion part as well as global Lipschitz continuity in the jump diffusion.

\subsection{A note on global asymptotic stability}
Stability concepts are essential in order to define computational accessible Gramians which are vital for identifying less relevant information in a system like \eqref{original_system}. We recall known facts for the linear part of \eqref{original_system}  based on the results in \cite{staboriginal}.
\begin{prop}\label{tmitkrostabe}
Let $f\equiv0$ and $B=0$ in \eqref{stochstatenew}, then
the following statements are equivalent: \begin{itemize}
\item[$(a)$]{{The state in \eqref{stochstatenew} is exponentially mean 
square stable, i.e., there are $k, \beta>0$, so that 
\begin{align*}
\sqrt{\mathbb E \left\|x(t, x_0, 0)\right\|^2_2}\leq \left\|x_0\right\|_2 k \expn^{-\beta t}.
\end{align*}}}
\item[$(b)$]{It holds that
\begin{align*}
\lambda\Big(I\otimes A+A\otimes 
I+\sum_{i, j=1}^d N_i\otimes N_j k_{ij}\Big)\subset \mathbb C_-,
\end{align*}
where $\lambda(\cdot)$ denotes the spectrum of a matrix.}
\item[$(c)$]{There exists a matrix $X>0$ with \begin{align*} 
A^\top X+X A+\sum_{i, j=1}^d N_i^\top X N_j k_{ij}<0.\end{align*}}
\end{itemize}
\end{prop}
\begin{proof}
A proof of these statements can be found in \cite{damm, redmannspa2}. 
\end{proof}
Throughout the rest of the paper, we assume that 
\begin{align}\label{shifted_stab}
\lambda\Big(I\otimes (A+c_1I)+(A+c_1 I)\otimes 
I+\sum_{i, j=1}^d N_i\otimes N_j k_{ij}\Big)\subset \mathbb C_-
\end{align}
for some constant $c_1$. According to Proposition \ref{tmitkrostabe} this means that \eqref{stochstatenew} with the shifted linear drift coefficient $A+c_1 I$ is mean square asymptotically stable for $B=0$ and $f\equiv 0$. The associated state variable is of the form $\expn^{c_1 t} x(t)$, so that the original state $x(t)$ ($B=0$ and $f\equiv 0$) needs to have a decay rate $\beta >c_1$, see Proposition \ref{tmitkrostabe} (a), given that $c_1$ is positive. We desire, but do not assume, that we can choose $c_1\geq c_f$, i.e., the decay rate of the linear part shall outperform the one-sided linear growth constant in \eqref{monotonicity_cond}. This requires a sufficiently stable linear part if $c_f>0$, e.g.,  for the nonlinearities in Example \ref{example1}. Since $c_f$ can also be negative, this means that the linear part of \eqref{stochstatenew} can even be exponentially increasing in some cases. 
Using classical arguments of \cite{staboriginal, mao} based on quadratic Lyapunov-type functions, we provide the following criterion for the global mean square stability of the uncontrolled state equation \eqref{stochstatenew}. This criterion is required around the discussion of the Gramians introduced later.
\begin{thm}\label{thm_global_stab}
Suppose that $B=0$ in \eqref{stochstatenew} and given constants $c_1, c_2\in \mathbb R$ and a matrix $X>0$. If we have 
\begin{align}\label{shifted_Lyap}
(A+c_1 I)^\top X+ X (A+c_1 I)+&\sum_{i, j=1}^d N_i^\top X N_j k_{ij}<0\quad \text{and}\\ \label{Xmonoton}
\langle x, X f(x)\rangle_2 &\leq c_2  \left\|X^{\frac{1}{2}}x\right\|_2^2                                                                                                                                                                                                                                                                                                                                                                                                                                                                                                                                                                                                                                                                                                                    \end{align}
for all $x\in \mathbb R^n$. Then, there exist constants $k, \beta>0$, so that 
\begin{align*}
\mathbb E \left\|x(t, x_0, 0)\right\|^2_2\leq \left\|x_0\right\|^2_2 k \expn^{(2(c_2-c_1)-\beta) t}.
\end{align*} 
\end{thm}
\begin{proof}
A proof is stated in Appendix \ref{appendixB}.
\end{proof}
\begin{remark}
If $c_1\geq c_2$ in Theorem \ref{thm_global_stab}, we obtain global mean square asymptotic stability for our nonlinear system. In particular, by assumption \eqref{monotonicity_cond}, \eqref{Xmonoton} holds for $X=I$ and $c_2=c_f$.  If \eqref{shifted_Lyap}
is now true for $X=I$ and $c_1=c_f$, mean square asymptotic stability follows.
\end{remark}

\section{Gramians and dominant subspace characterization}
\label{sec:reach}

In this section, algebraic objects, called Gramians, are introduced. We aim to construct them, so that their eigenspaces corresponding to small eigenvalues coincide with the information in \eqref{original_system} that can be neglected. It is not trivial to find the right notion for general nonlinearities $f$. However, the monotonicity condition in \eqref{monotonicity_cond} will become essential for our concept. In particular, positive (semi)definite Gramian candidates $X$ have to preserve \eqref{monotonicity_cond} in a certain sense when $\langle \cdot,  \cdot \rangle_2$  is replaced by $\langle \cdot, X \cdot \rangle_2$ or $\langle \cdot, X^{-1} \cdot \rangle_2$. 
We begin with a global Gramian concept to illustrate what we require. Subsequently, we immediately weaken it for practical reasons.  
\subsection{Monotonicity Gramians}
First, a pair of Gramians is defined that characterizes dominant subspaces of \eqref{original_system}  for all $u\in L^2_T$.
\begin{defn}\label{global_mon_gram}
Let $c_1$ and $c_2$ be constants. Then, a pair of matrices $(P, Q)$ with $P, Q>0$ is called global monotonicity Gramians if they satisfy
\begin{align}\label{newgram2}
 &(A+c_1 I)^\top P^{-1}+P^{-1}(A+c_1 I)+\sum_{i, j=1}^d N_i^\top P^{-1} N_j k_{ij} \leq -P^{-1}BB^\top P^{-1},\\ \label{newgram2Q}
 & (A+c_1 I)^\top Q+Q (A+c_1 I)+\sum_{i, j=1}^d N_i^\top Q N_j k_{ij} \leq -C^\top C,
                                       \end{align}
and if further holds that \begin{align}\label{PQ_monoton}
  \langle x,  P^{-1} f(x) \rangle_2 \leq c_2   \|P^{-\frac{1}{2}}x\|_2^2\quad \text{and}\quad    \langle x, Q f(x) \rangle_2 \leq c_2   \|Q^{\frac{1}{2}}x\|_2^2                                                                                                                                                                                                                                                                                                                                                                                                                                                                                                                                                                                                                                                                                                                       
                          \end{align}
for all $x\in\mathbb R^n$.
                                       \end{defn}
Notice that assumption \eqref{shifted_stab} ensures the existence of solutions to \eqref{newgram2} and \eqref{newgram2Q}, see \cite{bennerdammcruz, redmannspa2}. In the following, we state a sufficient criterion for the existence of Gramians also satisfying \eqref{PQ_monoton}.
\begin{prop}\label{propzugram2}
Suppose that \eqref{shifted_Lyap} and
\eqref{Xmonoton} hold with some constants $c_1$ and $c_2$. Then, global monotonicity Gramians $P$ and $Q$ exist with the same constants.
 \begin{proof}
We denote the left hand side of \eqref{shifted_Lyap} by $-Y$ and multiply it with $\gamma>0$. Hence, we have\begin{align} 
\label{eqtogetgram2}
(A+c_1 I)^\top (\gamma X)+(\gamma X) (A+c_1 I)+\sum_{i, j=1}^d N_i^\top (\gamma X) N_j k_{ij}=-\gamma Y.\end{align}  
Since $Y>0$, we can ensure that $-\gamma Y \leq -(\gamma X)BB^\top (\gamma X)$ if $\gamma$ is sufficiently small. Therefore, $P=(\gamma 
X)^{-1}$ solves \eqref{newgram2} for a potentially small $\gamma$. On the other hand, this $P$ gives us $\langle x,  P^{-1} f(x) \rangle_2= \gamma\langle x,  X f(x) \rangle_2\leq \gamma c_2   \|X^{\frac{1}{2}}x\|_2^2= c_2   \|P^{-\frac{1}{2}}x\|_2^2$. Now, we know that $-\gamma Y \leq -C^\top C$ if $\gamma$ is sufficiently large. Consequently, $Q=\gamma 
X$ satisfies \eqref{newgram2Q} for a potentially large $\gamma$. Moreover, we find that $ \langle x, Q f(x) \rangle_2= \gamma \langle x, X f(x) \rangle_2 \leq \gamma  c_2   \|X^{\frac{1}{2}}x\|_2^2= c_2   \|Q^{\frac{1}{2}}x\|_2^2$ using \eqref{Xmonoton}. This concludes the proof.
 \end{proof}
\end{prop}
\begin{remark}
Certainly, the existence of global monotonicity Gramians is not sufficient for our considerations. As we will see later, it is important to find candidates $P$ and $Q$ that have a large number of small eigenvalues. Consequently, one might have to solve a problem of minimizing $\trace(P)$ and $\trace(Q)$ subject to \eqref{newgram2}, \eqref{newgram2Q} and \eqref{PQ_monoton}. Moreover, we allow $c_1<c_2$  in Definition \ref{global_mon_gram} to have an additional degree of freedom. However, this comes with a price. We will observe that $c_2-c_1$ is supposed to be small. In fact, we desire to choose $c_1=c_2$ if such a $c_1$ ensures \eqref{shifted_stab}.  
\end{remark}
\begin{example}\label{example2}
\begin{itemize}
 \item 
Choosing $f=f^{(3)}$ from Example \ref{example1}, we see that $\langle x, X f^{(3)}(x)\rangle_2 \leq \|X^{\frac{1}{2}}x\|_2^2$ for any $X\geq 0$ and all $x\in\mathbb R^n$. Therefore, any solutions of \eqref{newgram2} and \eqref{newgram2Q} with $c_1=c_2=c_f=1$ are global monotonicity Gramians. In particular, we can choose the solution to the equality in \eqref{newgram2Q} and the candidate with minimal trace in \eqref{newgram2}.
\item If $f$ is globally Lipschitz in some norm, then there exist a Lipschitz constant $c_L$, so that  $\langle x, X f(x)\rangle_2 = \langle  X^{\frac{1}{2}} x, X^{\frac{1}{2}} f(x)\rangle_2  \leq \|X^{\frac{1}{2}}x\|_2 \|X^{\frac{1}{2}}f(x)\|_2\leq c_L \|X^{\frac{1}{2}}x\|_2^2$ given that $X=P^{-1}, Q>0$ meaning that every positive solution to  \eqref{newgram2} and \eqref{newgram2Q} can be picked. However, $c_L$ depends on $X$ which shows that $c_1$ and $c_2$ influence each other. On the other hand, this $c_L$ might not be the optimal candidate for the one-sided Lipschitz constant $c_2$ which can even be negative, i.e., it is also challenging to identify optimal constants.
\end{itemize}
\end{example}
We emphasize further that, generally, we cannot derive $P$ and $Q$ independent of \eqref{PQ_monoton}. For instance, fixing  $c_1=c_2\geq c_f$, 
we can easily find a solution $Q$ for \eqref{newgram2Q} and a vector $x\in\mathbb R^n$, so that 
$\langle x, Q f^{(1)}(x)\rangle_2 > c_2\|Q^{\frac{1}{2}}x\|_2^2$. Here, $f=f^{(1)}$ is the function defined in Example \ref{example1}. Having in mind that we aim to fix $c_1$ and $c_2$ close to each other with associated Gramians $P$ and $Q$ having a large number of small eigenvalues, the concept of global Gramians might generally be too restrictive. Therefore, it is more reasonable to seek for solutions of \eqref{newgram2} and \eqref{newgram2Q} that satisfy \eqref{PQ_monoton} on average instead of point-wise. This means, we aim to allow for  positive values of the monotonicity gaps \begin{align}\label{loc_max_functions}
G_{P^{-1}}(x) :=  \langle x, P^{-1} (f(x)-c_2 x)\rangle_2   \quad\text{and} \quad G_Q(x) :=  \langle x, Q (f(x)-c_2 x)\rangle_2                                                                                                                                                                                                                                                                                                                                                                                                                                                                                                                                                                                                                                                                                        \end{align}
as long as $G_{P^{-1}}$ and $G_Q$ are mainly non-positive on the essential parts of $\mathbb R^n$. We specify the above arguments in the following definition. In this context, we introduce the set $\mathcal U$ of controls $u\in L^2_T$ for which we desire to evaluate system \eqref{original_system}.  The following pair of Gramians $(P, Q)$ identifies less important direction for controls in $\mathcal U$. Therefore, it is meaningful to pick Gramian candidates that ensure a large set $\mathcal U$. 
\begin{defn}\label{def_av_mon_Gram}
Let $c_1, c_2$ be constants and $\mathcal U\subseteq L^2_T$ be the set of controls we are interested in. Then, a pair of matrices $(P, Q)$ with $P, Q>0$ is called average monotonicity Gramians for $\mathcal U$ if 
\eqref{newgram2} and \eqref{newgram2Q} are satisfied, respectively,
 and if instead of \eqref{PQ_monoton} it holds that  \begin{align}\label{P_average_monoton}
 \mathbb E\int_0^t \langle x(s),  P^{-1} f(x(s)) \rangle_2 ds &\leq c_2 \, \mathbb E\int_0^t  \|P^{-\frac{1}{2}}x(s)\|_2^2 ds\quad \text{and}\\ \label{Q_average_monoton}
 \mathbb E\int_0^t \langle x(s), Q f(x(s)) \rangle_2 ds &\leq c_2 \,\mathbb E\int_0^t \|Q^{\frac{1}{2}}x(s)\|_2^2 ds                                                                                                                                                                                                                                                                                                                                                                                                                                                                                                                                                                                                                                                                                                                      
                          \end{align}
for all $t\in [0, T]$ and all state variables $x(t)=x(t, 0, u)$ with $u\in \mathcal U$.
                                       \end{defn}
Certainly, a global is also an average monotonicity Gramian with $\mathcal U= L^2_T$. Suppose that there are areas, where one of the functions in \eqref{loc_max_functions} is positive. Then, controls $u$ concentrating the state variable $x$ in such areas for a long time will violate \eqref{P_average_monoton} or \eqref{Q_average_monoton}.

\begin{remark}\label{rem_gen_gram}
In Definitions \ref{global_mon_gram} and \ref{def_av_mon_Gram}, Gramians are constructed as solutions to (shifted) linear matrix inequalities in order to allow a practical computation. This is possible due to the monotonicity condition for $f$ in \eqref{monotonicity_cond} which shall be preserved in some sense under the inner products defined by the Gramians $P$ and $Q$. A more general version of global monotonicity Gramians is obtained by adding twice the estimates in \eqref{PQ_monoton} to \eqref{newgram2} and \eqref{newgram2Q} resulting in
\begin{align}\label{generalized_mon_Gram_P}
 &\hspace{-0.23cm} x^\top \Big( A^\top P^{-1}+P^{-1}A+\hspace{-0.05cm}\sum_{i, j=1}^d N_i^\top P^{-1} N_j k_{ij} \Big)x + 2 \langle x,  P^{-1} f(x) \rangle_2\leq -\| B^\top P^{-1} x\|_2^2 +c\|P^{-\frac{1}{2}}x\|_2^2,\\ \label{generalized_mon_Gram_Q}
  &\hspace{-0.23cm} x^\top \Big(A^\top Q+QA+\hspace{-0.05cm}\sum_{i, j=1}^d N_i^\top Q N_j k_{ij} \Big)x  + 2 \langle x, Q f(x) \rangle_2 \leq -\| Cx\|_2^2+c\|Q^{\frac{1}{2}}x\|_2^2
             \end{align}
for all $x\in\mathbb R^n$, where $c\geq 0$ is some ``small'' constant. The same way, average monotonicity Gramians can be generalized setting $x=x(s)$ in \eqref{generalized_mon_Gram_P} and \eqref{generalized_mon_Gram_Q}, taking the expected value and integrating both sides of these inequalities over each subinterval $[0, t]$ with $0<t\leq T$. However, we will not discuss this generalization in detail below.
\end{remark}

\subsection{Relevance of monotonicity Gramians}

In the following, we state in which sense the Gramians of Definition \ref{def_av_mon_Gram} help to identify the dominant subspaces of \eqref{original_system}. This then motivates a truncation procedure resulting in a special type of reduced system \eqref{red_system}.
Below, let us assume that $x_0=0$, i.e., $x(t) = x(t, 0, u)$. By definition, Gramians are positive (semi)definite matrices. Consequently, we can find an orthonormal basis  $(p_k)$ for $\mathbb R^n$ consisting of eigenvalues of $P$ with corresponding eigenvalues $(\lambda_{P, k})$. The same is true for $Q$, where the basis is denoted by $(q_k)$ with associated eigenvalues $(\lambda_{Q, k})$. Hence, the state variable can be represented as \begin{align}\label{eigen_rep}
x(t)= \sum_{k=1}^n \langle x(t), p_{k}  
\rangle_2 \,p_k\quad \text{and}\quad  x(t)= \sum_{k=1}^n \langle x(t), q_{k}  
\rangle_2 \,q_k.                                                                                                                                                                                                                                                                                                                                                                                                                                                                                                                                                                                                                                                                                                                                                                                                                                                                                                                                                                                                                                                          \end{align}
Based on this representation, we aim to answer which directions $p_k$ are less relevant in \eqref{stochstatenew} and which directions $q_k$ can be neglected in \eqref{output_eq}.
\begin{thm}\label{energy_est}
Let $P$ and $Q$ be average monotonicity Gramians for the set of controls $\mathcal U\subseteq L^2_T$ and constants $c_1, c_2$ according to Definition \ref{def_av_mon_Gram}. Moreover, let $(p_k, \lambda_{P, k})$ and $(q_k, \lambda_{Q, k})$ be associated bases of eigenvectors giving us \eqref{eigen_rep}. Then, given a zero initial state, we have
\begin{align}\label{P_est}
\sup_{t\in[0, T]}\mathbb E \langle x(t), p_{k}  \rangle_2^2 &\leq \lambda_{P, k} \expn^{c T} \left\|u\right\|_{L^2_T}^2,\\ \nonumber
\mathbb E \int_0^t\left\|y(s)\right\|_2^2ds&\leq 2\mathbb E\int_0^t \langle Q x(s), B u(s) \rangle_2 \expn^{c (t-s)} ds \\ \label{Q_est}
&= 2\sum_{k=1}^n \lambda_{Q, k} \mathbb E\int_0^t \langle q_k, x(s) \rangle_2 \langle q_k, B u(s) \rangle_2 \expn^{c (t-s)} ds
\end{align}
for all $t\in [0, T]$ and $u\in\mathcal U$, where $c=\max\{0, 2(c_2-c_1)\}$.
\end{thm}
\begin{proof}
We find inequalities for $\mathbb E\left[x(t)^\top X x(t)\right]$, where $X\in\{P^{-1}, Q\}$. To do so, we apply Lemma \ref{lemstochdiff} to $X^{\frac{1}{2}}x(t)$ and obtain \begin{align*}
 \frac{d}{dt}\mathbb E\left[x(t)^\top X x(t)\right]&= 2 \mathbb E\left[x(t)^\top X [Ax(t)+Bu(t)+ f(x(t))]\right]  + \sum_{i, j=1}^d \mathbb E\left[x(t)^\top N_i^\top X N_j x(t)\right]k_{ij}.
 \end{align*}
We integrate this equation over $[0, t]$ with $t\leq T$ yielding
\begin{align}\nonumber
\mathbb E\left[x(t)^\top X x(t)\right]&=
\mathbb E\int_0^t\bigg[x(s)^\top\Big(A^\top X+ X A + \sum_{i, j=1}^d  N_i^\top X N_j k_{ij}\Big) x(s)\bigg]ds\\ \nonumber
&\quad+ 2\int_0^t \mathbb E\left[x(s)^\top X \big[B u(s) +f(x(s))\big]\right] ds\\ \nonumber
&\leq \mathbb E\int_0^t\bigg[x(s)^\top\Big((A+c_1 I)^\top X+ X (A+c_1 I) + \sum_{i, j=1}^d  N_i^\top X N_j k_{ij}\Big) x(s)\bigg]ds \\ \label{estimate-mono_gram}
&\quad+ 2\int_0^t \mathbb E\left[x(s)^\top X B u(s) \right] ds + c\int_0^t \mathbb E\left[x(s)^\top X x(s) \right] ds
\end{align}
exploiting \eqref{P_average_monoton}, \eqref{Q_average_monoton} and that $x_0=0$. Setting $\alpha(t):=2\int_0^t \mathbb E\left[x(s)^\top X B u(s) \right] ds $ and $X=Q$, we obtain \begin{align*}
\mathbb E\left[x(t)^\top Q x(t)\right]&\leq -\int_0^t\mathbb E\left[x(s)^\top C^\top C x(s)\right]ds + 2\int_0^t \mathbb E\left[x(s)^\top Q B u(s) \right] ds+c\int_0^t \mathbb E\left[x(s)^\top Q x(s) \right] ds\\
&=-\left\|y\right\|_{L^2_t}^2+\alpha(t)+c\int_0^t \mathbb E\left[x(s)^\top Q x(s) \right] ds
\end{align*}
using \eqref{newgram2Q}. Therefore, by \eqref{gronwall2}, we have 
\begin{align*}
   \mathbb E\left[x(t)^\top Q x(t)\right]\leq  \int_0^t (\dot\alpha(s)-\mathbb E\left\|y(s)\right\|_2^2)\expn^{ c (t-s)}ds,
    \end{align*}
    and hence $\int_0^t \mathbb E\left\|y(s)\right\|_2^2 ds\leq  \int_0^t \dot\alpha(s)\expn^{c (t-s)}ds$. Inserting the representation for $x(s)$ in \eqref{eigen_rep} yields
\begin{align*}
\left\|y\right\|_{L^2_t}^2&\leq 2\int_0^t \mathbb E\left[x(s)^\top Q B u(s) \right]\expn^{c (t-s)} ds = 2\int_0^t \mathbb E\left[\Big(Q \sum_{k=1}^n \langle x(s), q_{k}  
\rangle_2 \,q_k\Big)^\top B u(s) \right]\expn^{c (t-s)} ds\\
&= 2 \sum_{k=1}^n \lambda_{Q, k} \int_0^t \mathbb E\left[\langle x(s), q_{k}  
\rangle_2 \;q_k^\top B u(s) \right]\expn^{c (t-s)} ds
\end{align*}
leading to \eqref{Q_est}. With $X=P^{-1}$ in \eqref{estimate-mono_gram}, it holds that \begin{align*}
\mathbb E\left[x(t)^\top P^{-1} x(t)\right]&\leq  -\int_0^t\mathbb E\left[x(s)^\top P^{-1}BB^\top P^{-1} x(s)\right]ds + 2\int_0^t \mathbb E\left[x(s)^\top P^{-1} B u(s) \right] ds \\
&\quad+c\int_0^t \mathbb E\left[x(s)^\top P^{-1} x(s) \right] ds\\
&= \mathbb E \int_0^t \left\|u(s)\right\|_2^2 - \|B^\top P^{-1} x(s)-u(s)\|_2^2ds+c\int_0^t \mathbb E\left[x(s)^\top P^{-1} x(s) \right] ds
\end{align*}
exploiting \eqref{newgram2}. Applying \eqref{gronwall1}, we obtain
\begin{align*}
\mathbb E\left[x(t)^\top P^{-1} x(t)\right]&\leq \mathbb E \int_0^t \left\|u(s)\right\|_2^2 ds+\int_0^t \int_0^s \mathbb E \left\|u(v)\right\|_2^2 dv\, c \expn^{c(t-s)} ds\\
&\leq \expn^{c t} \mathbb E \int_0^t \left\|u(s)\right\|_2^2 ds.
\end{align*}
We further observe that
\begin{align*}
 \langle x(t), p_{k}  \rangle_2^2 &\leq \lambda_{P, k}\; \sum_{i=1}^n \lambda_{P, i}^{-1} \langle x(t), p_{i}  
\rangle_2^2 =\lambda_{P, k} \Big\|\sum_{i=1}^n \lambda_{P, i}^{-\frac{1}{2}} \langle x(t), p_{i}  \rangle_2 \;p_{i}\Big\|_{2}^2
 =\lambda_{P, k} \Big\|P^{-\frac{1}{2}} x(t)\Big\|_{2}^2\\
 &= \lambda_{P, k} \; x(t)^\top P^{-1} x(t),
\end{align*}
so that \eqref{P_est} follows. This concludes the proof.
\end{proof}
Estimate \eqref{P_est} tells us that the state variable is small in the direction of $p_k$ if $\lambda_{P, k}$ is small and in case $c\,T$ is not too large ($c_2-c_1$ is supposed to be little). Consequently, these eigenspaces of $P$ can be neglected in our considerations. The eigenspaces spanned by vectors $q_k$ that are associated to small eigenvalues of $Q$ are also of minor relevance due to \eqref{Q_est}. This inequality shows that such $q_k$ barely contribute to the energy of the output $y$ on each subinterval $[0, t]$.
\begin{remark}\begin{itemize}
\item Following basically the same steps, the result of Theorem \ref{energy_est} holds also true if the more general notion of Gramians in Remark \ref{rem_gen_gram} is used.
\item
Theorem \ref{energy_est} is formulated for $u\in\mathcal U$ since it is based on \eqref{P_average_monoton} and \eqref{Q_average_monoton}. This does not mean that a reduced order model based on  neglecting eigenspaces of $P$ and $Q$ associated to small eigenvalues leads to a bad approximation for $u\in L^2_T\setminus \mathcal U$. This is because \eqref{P_average_monoton} and \eqref{Q_average_monoton} might still almost hold in that cases since suitable Gramians lead to $G_Q$ and $G_{P^{-1}}$ in \eqref{loc_max_functions} being small when they are positive. Then, the estimates in Theorem \ref{energy_est} will approximately hold.
\end{itemize}
\end{remark}

\subsection{Computation of monotonicity Gramians}\label{Sec_comp_gram}

We aim to compute average monotonicity Gramians $P$ and $Q$ for a large set $\mathcal U$ of controls. We choose them as solutions to \eqref{newgram2} and \eqref{newgram2Q}, so that $G_{P^{-1}}$ and $G_Q$ in \eqref{loc_max_functions} have a local maximum in the origin or a saddle point with very few increasing directions. Else, we might have several cases in which the monotonicity condition is immediately violated. This would not allow \eqref{P_average_monoton} and \eqref{Q_average_monoton} to hold for a large $\mathcal U$. On the other hand, it is essential that the area where the monotonicity condition is fulfilled clearly dominates the one where it does not hold. A possible and acceptable scenario in dimension $n=2$ is illustrated in Figure \ref{pic_mon_gap}. Here, the monotonicity gap $G_Q$ is depicted for $f=f^{(2)}$, $c_2=c_f=1$ and $Q=\smat  0.49426 &  0.58159\\
   0.58159  & 0.68542\srix$, a matrix with a large and a small eigenvalue. The blue color stands for small absolute values and red for large ones. $G_Q$ is non positive except for the black areas, where the monotonicity condition is slightly violated.
\begin{figure}[ht]
\center
\includegraphics[width=8cm,height=5.5cm]{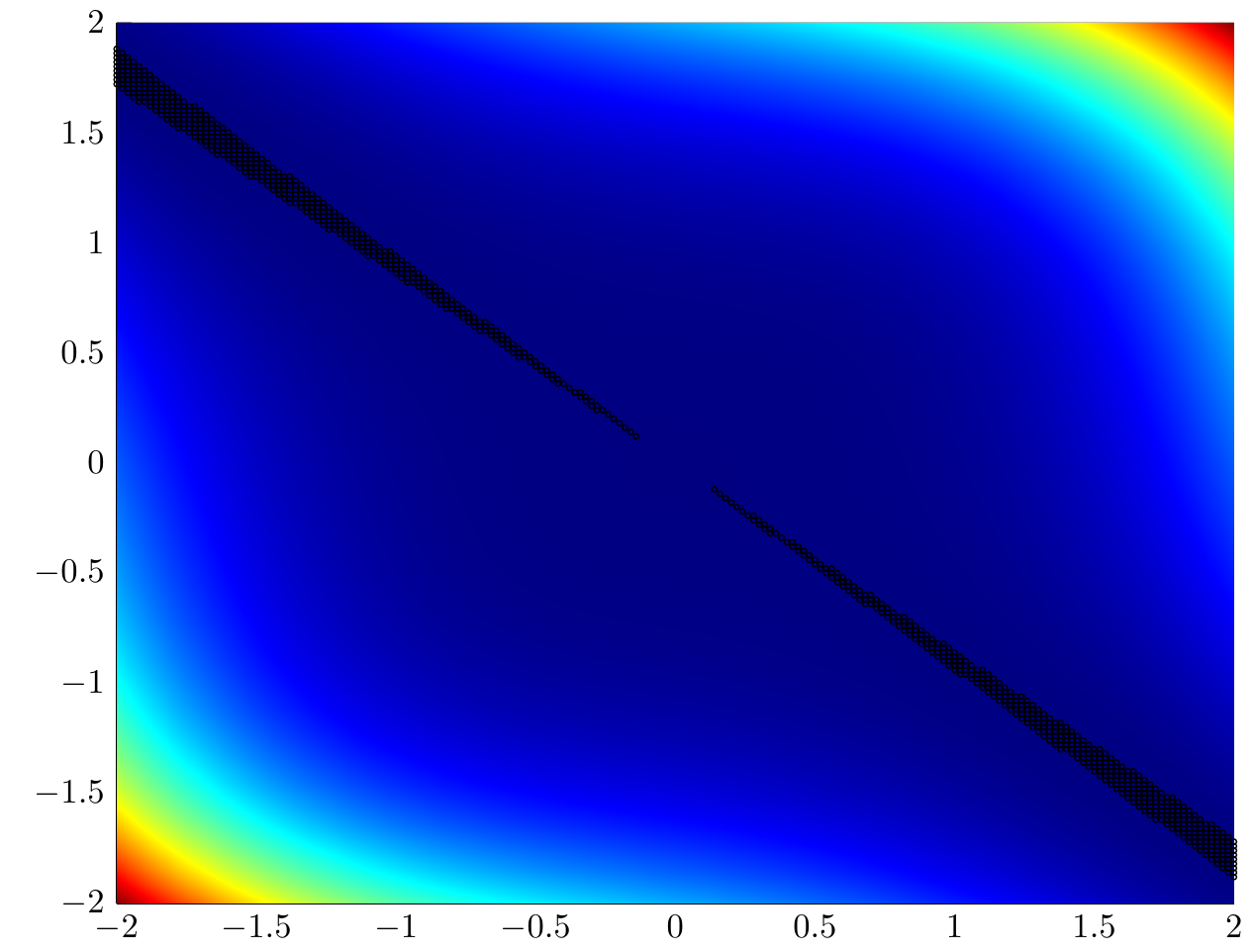}
\caption{$G_Q$ for a special choice of $Q$, $n=2$, $f=f^{(2)}$ and $c_2=c_f=1$. The area in black marks the regions, where $G_Q$ is positive.}\label{pic_mon_gap}
\end{figure}
In the following proposition, a simple criterion for local optimality for $G_{P^{-1}}$ and $G_Q$ is given.
\begin{prop}
Define the function $g(x) = \langle x, X(f(x)-c_2 x)\rangle_2$ with a constant $c_2$, $f$ being twice differentiable and $X>0$. We assume that \begin{align}\label{cond_max}
    f_{x_j}(x)\vert_{x=0} -c_2 e_j = -\tilde c_2 e_j                                                                                                                                                                                                                                                                                        \end{align}
for all $j\in \{1, \dots, n\}$ and $\tilde c_2>0$, where $e_j$ is the $j$-th unit vector in $\mathbb R^n$. Then, $g$ has a local maximum in $x=0$.
\end{prop}
\begin{proof}
It is easy to check that $x=0$ is an extreme value since $g_{x_i}(x) = \langle e_i, X(f(x)-c_2 x) \rangle_2 + \langle x, X(f_{x_i}(x)-c_2 e_i) \rangle_2$ is zero at the origin. Moreover, we derive $g_{x_i x_j}(x) = \langle e_i, X(f_{x_j}(x)-c_2 e_j) \rangle_2 + \langle e_j, X(f_{x_i}(x)-c_2 e_i) \rangle_2+\langle x, X f_{x_i x_j}(x) \rangle_2$. Therefore, we find $\big(g_{x_i x_j}(0)\big)_{i, j=1, \dots, n} = -2 \tilde c_2 X<0$ which concludes the proof.
\end{proof}
Condition \eqref{cond_max} is, e.g., satisfied if polynomials are considered. We can therefore observe that $G_{P^{-1}}$ and $G_Q$ have a local maximum for the choices of $f$ given in Example \ref{example1} in case $c_2$ is sufficiently large. In particular, we fix $c_2 \geq c_f$, since this means that $G_{P^{-1}}$ and $G_Q$ are non-positive along the bases of eigenvectors used in \eqref{eigen_rep}. This is a consequence of assumption \eqref{monotonicity_cond}. Theorem \ref{energy_est} motivates to choose $c_1$, so that $c_2-c_1$ is a small positive number. If possible, we even set $c_1= c_2$ providing $c=0$. If $c_1>0$, the possibility of this choice also depends on weather \eqref{shifted_stab} is  satisfied. We then compute the solution to \eqref{newgram2} having a minimal trace and the solution to the equality in \eqref{newgram2Q}. This provides that $G_{P^{-1}}$ and $G_Q$ are non-positive on large parts of $\mathbb R^n$ for the particular functions introduced in Example \ref{example1} and only small positive values are taken on the other area. This leads to \eqref{P_average_monoton} and \eqref{Q_average_monoton} for a large $\mathcal U$. This is what we observe from numerical experiments, where $A$ is a discrete Laplacian.
Let us now briefly sketch how such a minimal trace monotonicity Gramian $P$ is computed. We reformulate \eqref{newgram2} by multiplying it with $P$ from the left and from the right leading to \begin{align}\label{tran_gram_eq}
 (A+c_1 I)P+P(A+c_1 I)^\top+BB^\top  +\sum_{i, j=1}^d P N_i^\top P^{-1} N_j k_{ij} P \leq 0.
 \end{align}
Since $\sum_{i, j=1}^d P N_i^\top k_{ij} P^{-1} N_j P = P \smat N_1^\top&\dots &N_d^\top\srix(K\otimes P^{-1}) \smat N_1^\top&\dots &N_d^\top\srix^\top P$, we obtain the following equivalent representation
 \begin{align}\label{LMI}
 \begin{bmatrix}
 {(A+c_1 I)P+P(A+c_1 I)^\top+BB^\top} & {P \smat N_1^\top&\dots &N_d^\top\srix } \\
 {\smat N_1^\top&\dots &N_d^\top\srix^\top P }& {-K^{-1}\otimes P }
 \end{bmatrix}\leq 0
\end{align}
for \eqref{tran_gram_eq} based on Schur complement conditions for the definiteness of a matrix. Here, we need to further assume that $K$ is invertible. Now, we can use a linear matrix inequality solver to find a solution to the minimization of $\trace(P)$ subject to \eqref{LMI} and $P>0$. In this paper, we use YALMIP and MOSEK \cite{mosek, yalmip} for an efficient computation of $P$.\smallskip

In general, a good choice for $P$ and $Q$ guaranteeing \eqref{P_average_monoton} and \eqref{Q_average_monoton} for many different controls always depends on the particular nonlinearity $f$. Therefore, no universal recommendation can be given here.

\subsection{Extension under one-sided Lipschitz continuity}\label{sec_lip_gram}
Many functions $f$ satisfying \eqref{monotonicity_cond}  are also one-sided Lipschitz continuous. However, we require an extended version of this continuity concept in the context of the error analysis in Section \ref{sec_error_analysis}. In detail the following  inequalities are supposed to hold:\begin{align}\label{one_sided_lip_cond}
  \langle x\pm z, f(x)\pm f(z) \rangle_2 \leq c_f   \left\|x\pm z\right\|_2^2,                                                                                                                                                                                                                                                                                                                                                                                                                                                                                                                                                                                                                                                                                                                          \end{align}
for all $x, z\in \mathbb R^n$ and a constant $c_f$. Condition \eqref{one_sided_lip_cond} will later inspire the extended definition of Gramians. Notice that one-sided Lipschitz continuity is defined with a minus in \eqref{one_sided_lip_cond} but we additionally ask for this property when replacing each minus by a plus. In this context, let us look at the functions of Example \ref{example1} again. We begin with $f^{(2)}$ and $f^{(3)}$ and show that \eqref{one_sided_lip_cond} is satisfied.
\begin{example}\label{example5}
Inserting $f^{(3)}(x)=x-\left\|x\right\|_2^2 x$ below yields
 \begin{align*}
\langle x\pm z,  f^{(3)}(x)\pm f^{(3)}(z)\rangle_2 =  \left\|x\pm z\right\|_2^2 - \langle x\pm z,  \left\|x\right\|_2^2 x \pm \left\|z\right\|_2^2 z\rangle_2.
\end{align*}
Now, we find that  \begin{align*}
 \langle x\pm z,  \left\|x\right\|_2^2 x \pm \left\|z\right\|_2^2 z\rangle_2 &= \left\|x\right\|_2^4+\left\|z\right\|_2^4\pm \langle x,  z\rangle_2 (\left\|x\right\|_2^2  + \left\|z\right\|_2^2)\geq \left\|x\right\|_2^4+\left\|z\right\|_2^4 -0.5(\left\|x\right\|_2^2  + \left\|z\right\|_2^2)^2\\
 &= 0.5(\left\|x\right\|_2^2  - \left\|z\right\|_2^2)^2\geq 0
\end{align*}
and hence \eqref{one_sided_lip_cond} holds with $c_f=1$ in case $f=f^{(3)}$. We obtain from $f^{(2)}(x)=x-x^{\circ 3}$ that\begin{align*}
\langle x- z,  f^{(2)}(x)- f^{(2)}(z)\rangle_2 =  \left\|x- z\right\|_2^2 - \langle x- z,  x^{\circ 3} - z^{\circ 3}\rangle_2.
\end{align*}
Since we have that \begin{align*}
 \langle x- z,  x^{\circ 3} - z^{\circ 3}\rangle_2&=\sum_{i=1}^n (x_i^4 +z_i^4 - z_i x_i^3-x_i z_i^3)=\sum_{i=1}^n (x_i-z_i)^2 (x_i^2 +z_i^2 + z_i x_i)\\
 &\geq \sum_{i=1}^n (x_i-z_i)^2 0.5 (x_i^2 +z_i^2 + 2 z_i x_i)\geq 0,
\end{align*}
we obtain $\langle x- z,  f^{(2)}(x)- f^{(2)}(z)\rangle_2 \leq  \left\|x- z\right\|_2^2$ and consequently the point symmetry of $f^{(2)}$ yields  \begin{align*}
\langle x+ z,  f^{(2)}(x)+ f^{(2)}(z)\rangle_2=\langle x- (-z),  f^{(2)}(x)- f^{(2)}(-z)\rangle_2 \leq  \left\|x- (-z)\right\|_2^2=\left\|x+z\right\|_2^2.
\end{align*}
Therefore, $c_f=1$ in \eqref{one_sided_lip_cond} for $f= f^{(2)}$.
\end{example}
As we will see below, $f^{(1)}$ is also one-sided Lipschitz but \eqref{one_sided_lip_cond} is not fulfilled if a plus is considered.
\begin{example} \label{example4}
Using $f^{(1)}(x) = (1+a) x^{\circ 2} - x^{\circ 3}-a x$ leads to
 \begin{align*}
\langle x-z,  f^{(1)}(x)-f^{(1)}(z)\rangle_2 &=  -a\left\|x-z\right\|_2^2+ \langle x-z,  (1+a)(x^{\circ 2}-z^{\circ 2})-(x^{\circ 3}-z^{\circ 3})\rangle_2.
\end{align*}
We obtain that 
 \begin{align*}
&\langle x-z,  (1+a)(x^{\circ 2}-z^{\circ 2})-(x^{\circ 3}-z^{\circ 3})\rangle_2 =\sum_{i=1}^n [(1+a) (x_i^3-z_i x_i^2-x_i z_i^2+z_i^3)-x_i^4+x_i z_i^3+z_i x_i^3-z_i^4]\\
&=\sum_{i=1}^n(x_i-z_i)^2[(1+a)(x_i+z_i)-x_i^2-z_i^2-x_i z_i]\leq \frac{(1+a)^2}{3} \left\|x-z\right\|_2^2
\end{align*}
exploiting that $(1+a)(x_i+z_i)-x_i^2-z_i^2-x_i z_i\leq  \frac{(1+a)^2}{3}$ for all $i\in\{1, \dots, n\}$. Therefore, we have 
 \begin{align*}
\langle x-z,  f^{(1)}(x)-f^{(1)}(z)\rangle_2 \leq  \frac{a^2-a+1}{3}\left\|x-z\right\|_2^2.
\end{align*}
We observe that the one-sided Lipschitz constant is different from the monotonicity constant in Example \ref{example1}. Moreover, we show that \eqref{one_sided_lip_cond} does not hold with a plus. Let $n=1$ and $c_f$ be an arbitrary constant. We fix $x=1$ and $z=\epsilon -1$ with $\epsilon>0$. We obtain 
\begin{align*}
\langle x+z,  f^{(1)}(x)+f^{(1)}(z)\rangle_2 &= \epsilon [-a\epsilon + (1+a) (1+(\epsilon-1 )^2)-(1+(\epsilon-1)^3)]\\
&= \epsilon [2(1+a)-\epsilon^3+(4+a)\epsilon^2- (5+3a)\epsilon]>c_f \epsilon^2 = c_f \left\|x+z\right\|_2^2,
\end{align*}
if $\epsilon$ is sufficiently small and $a>-1$.
\end{example}
Motivated by the one-sided Lipschitz continuity \eqref{one_sided_lip_cond}, a Gramian based inner product shall preserve this property leading to the  following extension of Definition \ref{global_mon_gram}.

\begin{defn}\label{one_side_Lip_gram}
Let $c_1$ and $c_2$ be constants. Then, a pair of matrices $(P, Q)$ with $P, Q>0$ is called global one-sided Lipschitz Gramians if they satisfy
\eqref{newgram2},  \eqref{newgram2Q}  and 
\begin{equation}\label{PQ_lipschitz}
\begin{aligned}
  \langle x+z,  P^{-1} (f(x)+f(z)) \rangle_2 &\leq c_2   \|P^{-\frac{1}{2}}(x+z)\|_2^2,\\   \langle x-z, Q (f(x)-f(z)) \rangle_2 &\leq c_2   \|Q^{\frac{1}{2}}(x-z)\|_2^2                                                                                                                                                                                                                                                                                                                                                                                                                                                                                                                                                                                                                                                                                                                       
                          \end{aligned}
                          \end{equation}
for all $x, z\in\mathbb R^n$.
\end{defn}
\begin{example}
Let $P, Q>0$ be solutions to \eqref{newgram2},  \eqref{newgram2Q} and $f$ be globally Lipschitz with $-f(x)= f(-x)$. Then, we can always construct  global one-sided Lipschitz Gramians, since for $X\in  \{P^{-1}, Q\}$ satisfying \eqref{newgram2} and \eqref{newgram2Q}, we have that \begin{align*}\langle X^{\frac{1}{2}}(x\pm z), X^{\frac{1}{2}} (f(x)\pm f(z)) \rangle_2 \leq  \|X^{\frac{1}{2}}(x\pm z)\|_2 \|X^{\frac{1}{2}}(f(x)\pm f(z))\|_2\leq c_2   \|X^{\frac{1}{2}}(x\pm z)\|_2^2\end{align*} for some suitable constant $c_2$.
\end{example}

If \eqref{PQ_lipschitz} is satisfied for $z=0$, $P$ and $Q$ are global monotonicity Gramians. We will see later that a reduced order model based on the Gramians introduced in Definition \ref{one_side_Lip_gram} will lead to error estimates for all controls $u\in L^2_T$. However, as in the global monotonicity Gramian case, it might be inefficient to choose a Gramian allowing to derive estimates for all $u$. The error analysis will show that it is actually enough to have \eqref{PQ_lipschitz} for large/essential sets of pairs $(x, z)\in\mathbb R^n\times\mathbb R^n$ in order to find a reasonable error criterion for a large number of different controls, i.e., the one-sided Lipschitz gaps \begin{equation}\label{PQ_lipschitz_gap}
\begin{aligned}
  G_{P^{-1}}^+(x, z)&:=\langle x+z,  P^{-1} (f(x)+f(z)) \rangle_2 - c_2   \|P^{-\frac{1}{2}}(x+z)\|_2^2,\\  
  G_{Q}^-(x, z)&:=\langle x-z, Q (f(x)-f(z)) \rangle_2 -c_2   \|Q^{\frac{1}{2}}(x-z)\|_2^2                                                                                                                                                                                                                                                                                                                                                                                                                                                                                                                                                                                                                                                                                                                       
                          \end{aligned}
                          \end{equation}
in \eqref{PQ_lipschitz} are mainly negative  but also small positive values will be allowed. We postpone the discussion of a weaker version of Definition \ref{one_side_Lip_gram} to Section \ref{sec_error_analysis}.
\begin{remark}
One-sided Lipschitz Gramians are again special solutions of linear matrix inequalities for reasons of accessibility. Analogue to Remark \ref{rem_gen_gram} this concept can be formulated more generally. Adding twice \eqref{PQ_lipschitz} to the respective inequality in \eqref{newgram2} and  \eqref{newgram2Q}  
 leads to \begin{align}\label{generalized_lip_Gram_P}
 &\hspace{-0.23cm} (x+z)^\top \Big( A^\top P^{-1}+P^{-1}A+\hspace{-0.05cm}\sum_{i, j=1}^d N_i^\top P^{-1} N_j k_{ij} \Big)(x+z) + 2 \langle x+z,  P^{-1} (f(x)+f(z)) \rangle_2\\ \nonumber
 &\quad\leq -\| B^\top P^{-1} (x+z)\|_2^2 +c\|P^{-\frac{1}{2}}(x+z)\|_2^2,\\ \label{generalized_lip_Gram_Q}
  &\hspace{-0.23cm} (x-z)^\top \Big(A^\top Q+QA+\hspace{-0.05cm}\sum_{i, j=1}^d N_i^\top Q N_j k_{ij} \Big)(x-z)  + 2 \langle x-z, Q (f(x)-f(z)) \rangle_2\\ \nonumber
  &\quad\leq -\| C(x-z)\|_2^2+c\|Q^{\frac{1}{2}}(x-z)\|_2^2
             \end{align}
for all $x, z\in\mathbb R^n$ with $c\geq 0$. We will see that this structure is what one requires  to achieve a suitable global error bound  for all $u\in L^2_T$. Notice that $z=0$ leads to \eqref{generalized_mon_Gram_P} and \eqref{generalized_mon_Gram_Q}, respectively. We will not discuss a definition of Gramians $P$ and $Q$ via \eqref{generalized_lip_Gram_P}  and \eqref{generalized_lip_Gram_Q} in further detail but will refer to them within the error analysis.
\end{remark}
Now, let us briefly discuss the existence of global one-sided Lipschitz Gramians.
\begin{prop}\label{lip_gram_exist}
Given a matrix $X>0$ satisfying \eqref{shifted_Lyap} for some constant $c_1$ and  
\begin{align*}
\langle x\pm z, X (f(x)\pm f(z))\rangle_2 \leq c_2  \|X^{\frac{1}{2}}(x\pm z)\|_2^2                                                                                                                                                                                                                                                                                                                                                                                                                                                                                                                                                                                                                                                                                                                    \end{align*}
for all $x, z\in \mathbb R^n$ and a constant $c_2$. Then, global one-sided Lipschitz Gramians exist with these constants.
\end{prop}
\begin{proof}
 The proof uses the same argument as in Proposition \ref{propzugram2} and is therefore omitted.
\end{proof}

Example \ref{example4} indicates that the  global one-sided Lipschitz Gramian $P$ might not be well-defined in case $f=f^{(1)}$.

\section{Particular reduced order model}\label{sec_BT}

We select a nonsingular
$S\in\mathbb{R}^{n\times n}$ that we use to simultaneously diagonalize Gramians $P$ and $Q$. This means that the bases of eigenvectors $(p_k)$ and $(q_k)$ in \eqref{eigen_rep} will be the canonical basis of $\mathbb R^n$. Consequently, 
by Theorem \ref{energy_est}, unimportant directions can be identified with components in the transformed state variable that are associated with small diagonal entries of the diagonalized Gramians. In particular, the transformation matrix defines the new state by $x_n=Sx$. Inserting this into \eqref{original_system}  leads to an equivalent stochastic
system with coefficients
\begin{align}\label{coef_trans}
  (A_n, B_n, f_n, N_{n,i}, C_n):=(SAS^{-1},SB, S f(S^{-1}\cdot), S N_iS^{-1},CS^{-1})
\end{align}
instead of the original ones $(A, B, f, N_i,C)$, i.e., \begin{align}\label{stochstatebal}
             dx_n(t) = [A_n x_n(t)+B_n u(t)+ f_n\left(x_n(t)\right)]dt+\sum_{i=1}^d N_{n,i}\left(x_n(t-)\right)dM_i(t),\quad
            y(t) = C_n x_n(t),
\end{align}
with $t\in [0, T]$ and $x_n(0)=0$.
The new system \eqref{stochstatebal} has the same input $u$ and output $y$. Moreover, 
properties like asymptotic stability are not affected. However, the Gramians are different. These are given in the following proposition, where the precise diagonalizing transformation is stated.
\begin{prop}\label{prop_bal}
 Suppose that $S$ is an invertible matrix. If $P$ and $Q$ are global/average monotonicity or one-sided Lipschitz Gramians of \eqref{original_system} according to Definitions \ref{global_mon_gram}, \ref{def_av_mon_Gram} or \ref{one_side_Lip_gram}. Then, $P_n= SPS^\top$ and $Q_n=S^{-\top}QS^{-1}$ are the respective Gramians in the transformed setting \eqref{stochstatebal}. Given that $P, Q>0$, we find that  $P_n=Q_n = \Sigma_n= \diag(\sigma_1,\ldots,\sigma_n) $ using the balancing transformation \begin{align}\label{bal_transform}
       S=\Sigma_n^{\frac{1}{2}} U^\top L_P^{-1},                                                                                                                                                                                                                                                                                                                                                                                                                                                                                                                                                                                                                                                                                                                                                                                          \end{align} 
where $P=L_PL_P^\top$ and $L_P^\top QL_P=U\Sigma_n^2 U^\top$ is a spectral factorization  with an orthogonal $U$.
\end{prop}
\begin{proof}
We multiply \eqref{newgram2} and \eqref{newgram2Q} with $S^{-\top}$ from the left and with $S^{-1}$ from the right hand side. Consequently, we see that
 $SPS^\top$ and $S^{-\top}QS^{-1}$ satisfy these inequalities under the coefficients in \eqref{coef_trans}.
Moreover, \eqref{PQ_monoton} is preserved under this transformation, since                                       
\begin{align*}
  \langle x,  P_n^{-1} f_n(x) \rangle_2 &= 
 \langle x,  S^{-\top} P^{-1} S^{-1} Sf(S^{-1}x) \rangle_2  = \langle S^{-1} x,  P^{-1} f(S^{-1}x) \rangle_2 
  \leq c_2   \|P^{-\frac{1}{2}}S^{-1}x\|_2^2\\
  &=c_2   \|P^{-\frac{1}{2}}_n x\|_2^2  \quad \text{and}\\
  \langle x, Q_n f_n(x) \rangle_2&= \langle x, S^{-\top} Q S^{-1} Sf(S^{-1}x) \rangle_2= \langle S^{-1}x,  Q f(S^{-1}x) \rangle_2 \leq c_2   \|Q^{\frac{1}{2}}S^{-1}x\|_2^2    = c_2   \|Q^{\frac{1}{2}}_nx\|_2^2.                                                                                                                                                                                                                                                                                                                                                                                                                                                                                                                                                                                                                                                                                                                      
                          \end{align*}
 Analogue, we can prove that the one-sided Lipschitz conditions \eqref{PQ_lipschitz} hold under the transformation. With $x_n(s)= x_n(s, 0, u)$  given $u\in\mathcal U$, we now find  \begin{align*}
  \langle x_n(s),  P_n^{-1} f_n(x_n(s)) \rangle_2 =  \langle x(s),  P^{-1} f(x(s)) \rangle_2\quad\text{and}\quad  \langle x_n(s),  Q_n f_n(x_n(s)) \rangle_2 =  \langle x(s),  Q f(x(s)) \rangle_2,
  \end{align*}
  as well as \begin{align*}
 \|P_n^{-\frac{1}{2}} x_n(s)\|_2^2 = \|P^{-\frac{1}{2}}x(s)\|_2^2\quad\text{and}\quad \|Q_n^{-\frac{1}{2}} x_n(s)\|_2^2 = \|Q^{\frac{1}{2}}x(s)\|_2^2,
             \end{align*}
so that the average monotonicity conditions \eqref{P_average_monoton} and \eqref{Q_average_monoton} still hold for the same set $\mathcal U$. We use \eqref{bal_transform} and obtain $P_n= \Sigma_n^{\frac{1}{2}} U^\top L_P^{-1} P L_P^{-\top} U \Sigma_n^{\frac{1}{2}}=\Sigma_n$ as well as $Q_n= \Sigma_n^{-\frac{1}{2}} U^\top L_P^{\top} Q L_P U \Sigma_n^{-\frac{1}{2}}=\Sigma_n$ which concludes the proof.
\end{proof}
We observe that the diagonal entries of the balanced Gramians are $\sigma_i= \sqrt{\lambda_i(PQ)}$. We call them Hankel singular values (HSVs) from now on.
Now, we partition the balanced state $x_n=\begin{bmatrix}
x_{n, 1}\\ x_{n,2}                                         
\end{bmatrix}$ and 
$\Sigma_n=\diag(\Sigma_r,\Sigma_{2, n-r})$, where
  $\Sigma_r= \diag(\sigma_1,\ldots,\sigma_r)$ contains the large and $\Sigma_{2, n-r}=\diag(\sigma_{r+1},\ldots,\sigma_n)$, $r<n$, the small HSVs. The same is done for \eqref{coef_trans} yielding
  \begin{equation}\label{part_bal}
  \begin{aligned}
A_n= \begin{bmatrix}{A}_{r}&\star\\ 
\star&\star\end{bmatrix},\quad B_n &= \begin{bmatrix}{B}_r\\\star\end{bmatrix},\quad N_{n,i}= \begin{bmatrix}{N}_{r,i}&\star\\ 
\star&\star\end{bmatrix}, \quad C_n= \begin{bmatrix}{C}_r &
\star\end{bmatrix}\quad \text{and}\\
f_r(x_r):&=\tilde f_r(\smat x_r\\ 0\srix), \quad \text{where}\quad f_n = \begin{bmatrix}{\tilde f}_r\\\star\end{bmatrix},\quad x_r\in \mathbb R^r,\quad 0\in \mathbb R^{n-r}. 
  \end{aligned}
  \end{equation}
Since $x_{n,2}$ is associated to small values in $\Sigma_{2, n-r}$, we truncate the equation for these variables and remove them from the dynamics of $x_{n, 1}$ and $y$. This results in a reduced system
 \eqref{red_system} with coefficients given by \eqref{part_bal}. Setting $V=V_r$ and $W=W_r$, where \begin{align*}                                                                                                S^{-1}=\begin{bmatrix}                                                                                
 V_r & \star                                                                                                   \end{bmatrix} \quad \text{and}\quad
S^\top=\begin{bmatrix}                                                                                
 W_r & \star                                                                                                   \end{bmatrix},                                                                                                   \end{align*}
we see that our reduced system's structure is of the form as in \eqref{red_mat_projection}. Here, $S$ is given by \eqref{bal_transform}.

\section{Error analysis of Gramian based reduced system}\label{sec_error_analysis}

We consider the reduced system \eqref{red_system} with state dimension $r$ and coefficients like in \eqref{part_bal}. As an intermediate step, let us introduce the same type of reduced model with dimension $k= r, r+1, \dots, n$ which we write as follows:
 \begin{align}\label{bal_par}
 dx_k(t) = [A_{k} x_k(t) +  B_k u(t)+ f_k(x_k(t))]dt + \sum_{i=1}^{d} N_{k, i} x_k(t-) dM_i(t),\quad
 y_{k}(t)= C_{k} x_k(t).
\end{align}
Setting $y_n:=y$, we then observe that 
\begin{align}\label{remove_HSVs_individually}
  \left\|y-y_{r}\right\|
 \leq \sum_{i=r+1}^n \left\| y_{k}-y_{k-1}\right\|,
 \end{align}
where $\left\| \cdot\right\|$ is some function space norm. This 
means that we have to investigate the error $ \left\| y_{k}-y_{k-1}\right\|$ of removing a single HSV. We can derive the reduced system of order $k-1$ from \eqref{bal_par} by setting the last entry of $x_k$ equal to zero. Doing so, we obtain \begin{equation}\label{bal_par_minus1}                                                                                                                                                                                                                                                                                                                                                                                                                                                                                                                        \begin{aligned}
 d\smat x_{k-1}(t)\\0\srix = &\big[A_{k} \smat x_{k-1}(t)\\0\srix  +  B_k u(t)+ f_k\Big(\smat x_{k-1}(t)\\0\srix \Big)-\smat 0 \\ v_0(t)\srix\big]dt\\
 &+ \sum_{i=1}^{d}\big[ N_{k, i} \smat x_{k-1}(t-)\\0\srix -\smat 0 \\ v_i(t-)\srix\big] dM_i(t),\quad
 y_{k-1}(t)= C_{k} \smat x_{k-1}(t)\\0\srix,
\end{aligned}
\end{equation}
where the first $k-1$ rows in the state equation of \eqref{bal_par_minus1}  represent the reduced order model of dimension $k-1$ and $v_0, \dots, v_d$ are (non specified) scalar processes that are introduced to ensure the equality in the last line which can be read as $d 0 = 0 dt + \sum_{i=1}^{d} 0 dM_i(t)$.
\begin{thm}\label{thm_error_bound}
 Let $y$ be the output of \eqref{original_system} with $x(0)=0$ and given the $r$-dimensional reduced system \eqref{red_system} with output $y_r$, coefficients as in \eqref{part_bal} and $x_r(0)=0$. If this reduced system is based on Gramians $P$ and $Q$ satisfying \eqref{newgram2} and \eqref{newgram2Q} for a constant $c_1$. Then,  for all $u\in L^2_T$, we have \begin{align*}
&\sqrt{\mathbb E \int_0^T  \left\|y(s)-y_{r}(s)\right\|_{2}^2 \expn^{c(T-s)} ds}\leq\\
& \sum_{k=r+1}^n  \sqrt{\mathbb E\int_0^T\left[2 G_{Q}^-\Big(V_k x_k(s), V_{k-1}x_{k-1}(s)\Big)+  \sigma_k^2\big(2 G_{P^{-1}}^+\Big(V_k x_k(s), V_{k-1}x_{k-1}(s)\Big) +4\left\| u(s)\right\|_{2}^2 \big)\right]\expn^{c(T-s)} ds}.                                                                                                                                                                                                                                                                                                                                                                                                                                                                                                                                                                                   \end{align*}
where $c=\max\{0, 2(c_2-c_1)\}$ is defined by another constant $c_2$ (e.g. the parameter of Definitions \ref{global_mon_gram}, \ref{def_av_mon_Gram} or \ref{one_side_Lip_gram}) and $G_{P^{-1}}^+$, $G_{Q}^-$ are the associated one-sided Lipschitz gaps in \eqref{PQ_lipschitz_gap}. Moreover, $x_k$ is the reduced state variable of order $k=r, r+1. \dots, n$ and $V_k$ is the associated projection matrix being the first $k$ columns of the inverse $S^{-1}$ of the balancing transformation defined by \eqref{bal_transform}.
\end{thm}
\begin{kor}\label{cor_bound}
 Given the assumptions of Theorem \ref{thm_error_bound}, let $P$ and $Q$ be global one-sided Lipschitz Gramians according to Definition \ref{one_side_Lip_gram}. Then, the following bound holds: \begin{align}\label{classical_bound}
\sqrt{\mathbb E \int_0^T  \left\|y(s)-y_{r}(s)\right\|_{2}^2 \expn^{c(T-s)} ds}\leq 
 2\sum_{k=r+1}^n \sigma_k  \sqrt{\mathbb E\int_0^T  \left\| u(s)\right\|_{2}^2 \expn^{c(T-s)} ds}                                                                                                                                                                                                                                                                                                                                                                                                                                                                                                                                                                                  \end{align}
 for all $u\in L^2_T$. The same bound is established if the Gramians are defined by \eqref{generalized_lip_Gram_P}  and \eqref{generalized_lip_Gram_Q}.
\end{kor}
\begin{proof}
The functions $G_{P^{-1}}^+$ and $G_{Q}^-$ are non positive by construction of the global one-sided Lipschitz Gramians. Consequently, the result immediately follows from the one of Theorem \ref{thm_error_bound}. It is not an immediate consequence of Theorem \ref{thm_error_bound} that \eqref{generalized_lip_Gram_P}  and \eqref{generalized_lip_Gram_Q} lead to the same result. However, the proof uses exactly the same ideas. Therefore, it is omitted.
\end{proof}
\begin{remark}\label{rem_error_discussion}
\begin{itemize}
 \item We found the classical bound for reduced order systems based on balanced truncation in Corollary \ref{cor_bound} up to the exponential terms in \eqref{classical_bound}, see \cite{BT_bound_enns, BT_bound_glover} for the deterministic and \cite{bennerdammcruz} for the stochastic linear case. As mentioned before, choices of Gramians are only acceptable if $c$ is sufficiently small, i.e., the exponentials do not dominate. On the other hand, global one-sided Lipschitz Gramians might not be a optimal in terms of their spectrum, so that a weaker concept is more reasonable.
\item
 As mentioned in Section \ref{sec_lip_gram}, we can allow for small positive one-sided Lipschitz gaps $G_{Q}^-$ and $G_{P^{-1}}^+$, see \eqref{PQ_lipschitz_gap}, in certain (small) regions. If we pick $P$ and $Q$ accordingly, Theorem \ref{thm_error_bound} then tells us that the averages \begin{align*}
&\mathbb E\int_0^T G_{Q}^-\Big(V_k x_k(s), V_{k-1}x_{k-1}(s)\Big)\expn^{c(T-s)} ds\quad\text{and}\\
&\mathbb E\int_0^T  G_{P^{-1}}^+\Big(V_k x_k(s), V_{k-1}x_{k-1}(s)\Big) \expn^{c(T-s)} ds                                                                                                                                                                                                                                                                                                                                                                                                                                                                                                                                                                                  \end{align*}
 will be non positive for a large number of controls $u\in L^2_T$ and slightly positive in many of the other scenarios. This means that \eqref{classical_bound} will (approximately) hold for many controls.
\item In case we have a priori information concerning the solution space of the system, we can say even more. This is given if $P$ and $Q$ are monotonicity Gramians according to Definitions \ref{global_mon_gram} or \ref{def_av_mon_Gram}, because of \eqref{P_est} in Theorem \ref{energy_est}. This estimate provides that we obtain a small state approximation error, i.e., $x(t)\approx V_k x_k(t)$ for $k\in\{r, \dots, n-1\}$, if the truncated HSVs $\sigma_{k+1}, \dots, \sigma_n$ are of low order. In particular, we have $V_{k+1}x_{k+1}(t)\approx V_k x_k(t)$ since this is the error of just removing $\sigma_{k+1}$. Therefore, we can conclude that we need $G_{Q}^-$ and $G_{P^{-1}}^+$ to be mainly negative solely on sets of pairs $(x, z)\in\mathbb R^n\times\mathbb R^n$ with $x\approx z$. In general, monotonicity Gramians do not ensure \eqref{classical_bound}, but due to the continuity of $f$, we can say that \begin{align*}
&\mathbb E\int_0^T G_{Q}^-\Big(V_k x_k(s), V_{k-1}x_{k-1}(s)\Big)\expn^{c(T-s)} ds\approx  \mathbb E\int_0^T G_{Q}^-\Big(V_k x_k(s), V_{k}x_{k}(s)\Big)\expn^{c(T-s)} ds=0,\\
&\mathbb E\int_0^T  G_{P^{-1}}^+\Big(V_k x_k(s), V_{k-1}x_{k-1}(s)\Big) \expn^{c(T-s)} ds  \approx \mathbb E\int_0^T \underbrace{G_{P^{-1}}^+\Big(V_k x_k(s), V_{k}x_{k}(s)\Big)}_{= \,4 G_{P^{-1}}\big(V_{k}x_{k}(s)\big)} \expn^{c(T-s)} ds.                                                                                                                                                                                                                                                                                                                                                                                                                                                                                                                                                                                 \end{align*}
Now, the monotonicity gap $G_{P^{-1}}$ defined in \eqref{loc_max_functions} is non positive on average for $u\in \mathcal U$ by construction of the average monotonicity Gramian $P$. This ensures that the bound of Corollary \ref{cor_bound} might still deliver a reasonable error criterion although it does not hold.
 \end{itemize}
\end{remark}
\begin{proof}[Proof of Theorem \ref{thm_error_bound}]
We introduce $x_-(t):= x_k(t) -\smat x_{k-1}(t)\\0\srix $ and $x_+(t):= x_k(t) +\smat x_{k-1}(t)\\0\srix $, for which the dynamics are obtained by subtracting/adding \eqref{bal_par} and \eqref{bal_par_minus1}, i.e., \begin{align} \label{eq_xminus}
 dx_-(t) &= [A_{k} x_-(t)+\smat 0 \\ v_0(t)\srix+ f_k(x_k(t))-f_k\big(\small{\smat x_{k-1}(t)\\0\srix} \big)]dt + \sum_{i=1}^{d}\big[ N_{k, i} x_-(t-) + \smat 0 \\ v_i(t-)\srix\big]dM_i(t),\\ \label{eq_xplus}
 dx_+(t) &= [A_{k} x_+(t)+2B_k u(t)-\smat 0 \\ v_0(t)\srix+ f_k(x_k(t))+f_k\big(\small{\smat x_{k-1}(t)\\0\srix} \big)]dt + \sum_{i=1}^{d}\big[ N_{k, i} x_+(t-) - \smat 0 \\ v_i(t-)\srix\big]dM_i(t).
\end{align} 
Recalling that $\Sigma_k=\diag(\sigma_1, \dots, \sigma_k)$ denotes the diagonal matrix of the $k$ largest HSVs of the original system, we know, by Proposition \ref{prop_bal}, that $\Sigma_n$ satisfies \eqref{newgram2} and \eqref{newgram2Q} with the balanced realization \eqref{coef_trans}. Evaluating the left upper $k\times k$ block of the equations associated to $\Sigma_n$, we obtain
\begin{align}
\label{aux_reach_gram}
(A_{k}+c_1 I)^\top \Sigma_k^{-1} + \Sigma_k^{-1} (A_{k}+c_1 I) 
+ \sum_{i, j=1}^d N_{k, i}^\top \Sigma_k^{-1}   N_{k, j} k_{ij} 
&\leq -\Sigma_k^{-1}  B_k B_k^\top \Sigma_k^{-1},\\
\label{aux_obs_gram}
(A_{k}+c_1 I)^\top \Sigma_k + \Sigma_k (A_{k}+c_1 I) + \sum_{i, j=1}^d  N_{k, i}^\top \Sigma_k N_{k, j} k_{ij}&\leq - C_{k}^\top C_{k}.
\end{align}
Taking \eqref{eq_xminus} into account, Lemma \ref{lemstochdiff} is applied to $\Sigma_k^{\frac{1}{2}}x_-(t)$ to obtain \begin{align*}
 \frac{d}{dt}\mathbb E\left[x_-(t)^\top \Sigma_k x_-(t)\right]= & 2 \mathbb E\left[x_-(t)^\top \Sigma_k [A_k x_-(t)+\smat 0 \\ v_0(t)\srix+ f_k(x_k(t))-f_k\big(\small{\smat x_{k-1}(t)\\0\srix}\big)]\right] \\
 &+ \sum_{i, j=1}^d \mathbb E\left[\big( N_{k, i} x_-(t) + \smat 0 \\ v_i(t)\srix\big)^\top \Sigma_k \big( N_{k, j} x_-(t) + \smat 0 \\ v_j(t)\srix\big)\right]k_{ij}.
 \end{align*}
Integrating this equation over $[0, t]$ with $t\leq T$ yields
\begin{align*}
\mathbb E\left[x_-(t)^\top \Sigma_k x_-(t)\right]&=
\mathbb E\int_0^t x_-(s)^\top\Big(A_k^\top \Sigma_k+ \Sigma_k A_k + \sum_{i, j=1}^d  N_{k, i}^\top \Sigma_k N_{k, j} k_{ij}\Big) x_-(s) ds\\
&\quad+ 2\mathbb E\int_0^t x_-(s)^\top \Sigma_k \big[ f_k(x_k(s))-f_k\big(\small{\smat x_{k-1}(s)\\0\srix}\big)\big] ds+R_-(t), \end{align*}
where $R_-(t)=\mathbb E\int_0^t 2 x_-(s)^\top \Sigma_k \smat 0 \\ v_0(s)\srix + \sum_{i, j=1}^d \left(2 N_{k, i} x_-(s) + \smat 0 \\ v_i(s)\srix\right)^\top \Sigma_k \smat 0 \\ v_j(s)\srix k_{ij}  ds$. Let $x_{k, 2}$ be the last entry of $x_k$ and hence also of $x_-$. Moreover, $n_{k, i}$ shall denote the last line of $N_{k, i}$.
Therefore, we obtain that $x_-(s)^\top \Sigma_k \smat 0 \\ v_0(s)\srix = \sigma_k x_{k, 2}(s) v_0(s)$ and $\left(2 N_{k, i} x_-(s) + \smat 0 \\ v_i(s)\srix\right)^\top \Sigma_k \smat 0 \\ v_j(s)\srix k_{ij}   = \sigma_k \left(2 n_{k, i} x_-(s) + v_i(s)\right) v_j(s)k_{ij}$. By construction of $v_i$ in \eqref{bal_par_minus1}, we have
$-2n_{k, i}  \smat x_{k-1}(s)\\0\srix+2v_i(s)=0$, so that $\sigma_k \left(2 n_{k, i} x_-(s) + v_i(s)\right) v_j(s)k_{ij}=\sigma_k \left(2 n_{k, i} x_k(s) - v_i(s)\right) v_j(s)k_{ij}$.  Therefore, it holds that \begin{align*}
 R_-(t)\leq \sigma_k\mathbb E\int_0^t 2 x_{k, 2}(s) v_0(s) + \sum_{i, j=1}^d \left(2 n_{k, i} x_k(s) + v_i(s)\right) v_j(s)k_{ij} ds                                                                                                                                                                                                                                                                     \end{align*}
exploiting that $\sum_{i, j=1}^d v_i(s)v_j(s) k_{ij} \geq 0$, because $K=(k_{ij})$ is positive semidefinite. Hence, \begin{align*}
\mathbb E\left[x_-(t)^\top \Sigma_k x_-(t)\right]&\leq
\mathbb E\int_0^t x_-(s)^\top\Big((A_k+c_1 I)^\top \Sigma_k+ \Sigma_k (A_k+c_1 I) + \sum_{i, j=1}^d  N_{k, i}^\top \Sigma_k N_{k, j} k_{ij}\Big) x_-(s)ds\\
&\quad+ 2\mathbb E\int_0^t x_-(s)^\top \Sigma_k \big[ f_k(x_k(s))-f_k\big(\small{\smat x_{k-1}(s)\\0\srix}\big)-c_2 x_-(s)\big] ds\\
&\quad+\sigma_k\mathbb E\int_0^t 2 x_{k, 2}(s) v_0(s) + \sum_{i, j=1}^d \left(2 n_{k, i} x_k(s) + v_i(s)\right) v_j(s)k_{ij} ds\\
&\quad+  c\int_0^t \mathbb E\left[x_-(s)^\top \Sigma_k x_-(s) \right] ds. \end{align*}
We set $\mathcal T_{k, -}(t):=2\mathbb E\int_0^t x_-(s)^\top \Sigma_k \big[ f_k(x_k(s))-f_k\big(\small{\smat x_{k-1}(s)\\0\srix}\big)-c_2 x_-(s)\big]ds$ and $\alpha_k(t):=\mathbb E\int_0^t 2 x_{k, 2}(s) v_0(s) + \sum_{i, j=1}^d \left(2 n_{k, i} x_k(s) + v_i(s)\right) v_j(s)k_{ij} ds$. Based on \eqref{aux_obs_gram} combined with the definitions of the outputs in \eqref{bal_par} and \eqref{bal_par_minus1}, we have \begin{align*}
\mathbb E\left[x_-(t)^\top \Sigma_k x_-(t)\right]\leq
- \left\|y_k-y_{k-1}\right\|_{L^2_t}^2+\mathcal T_{k, -}(t) +\sigma_k\alpha_k(t)+  c\int_0^t \mathbb E\left[x_-(s)^\top \Sigma_k x_-(s) \right] ds. \end{align*}
We obtain by \eqref{gronwall2} that \begin{align}\label{inter_result}
 \mathbb E \int_0^t  \left\|y_k(s)-y_{k-1}(s)\right\|_{2}^2 \expn^{c(t-s)} ds\leq   \int_0^t  \big(\dot{\mathcal T}_{k, -}(s) +\sigma_k\dot \alpha_k(s) \big)\expn^{c(t-s)} ds.                                
    \end{align}
Now, exploiting  Lemma \ref{lemstochdiff} for the process $\Sigma_k^{-\frac{1}{2}}x_+(t)$ together with \eqref{eq_xplus} yields
\begin{align*}
\mathbb E\left[x_+(t)^\top \Sigma_k^{-1} x_+(t)\right]&=
\mathbb E\int_0^t x_+(s)^\top\Big(A_k^\top \Sigma_k^{-1}+ \Sigma_k^{-1} A_k + \sum_{i, j=1}^d  N_{k, i}^\top \Sigma_k^{-1} N_{k, j} k_{ij}\Big) x_+(s)ds\\
&\quad+ 2\mathbb E\int_0^t x_+(s)^\top \Sigma_k^{-1} \big[ f_k(x_k(s))+f_k\big(\small{\smat x_{k-1}(s)\\0\srix}\big)\big]ds\\
&\quad +\mathbb E  \int_0^t 4 x_+(s)^\top \Sigma_k^{-1} B_k u(s) ds-R_+(t), \end{align*}
where $R_+(t)=\mathbb E\int_0^t 2 x_+(s)^\top \Sigma_k^{-1} \smat 0 \\ v_0(s)\srix + \sum_{i, j=1}^d \left(2 N_{k, i} x_+(s) - \smat 0 \\ v_i(s)\srix\right)^\top \Sigma_k^{-1} \smat 0 \\ v_j(s)\srix k_{ij}  ds$. We observe that $x_+(s)^\top \Sigma_k^{-1} \smat 0 \\ v_0(s)\srix = \sigma_k^{-1} x_{k, 2} v_0(s)$ and $\left(2 N_{k, i} x_+(s) - \smat 0 \\ v_i(s)\srix\right)^\top \Sigma_k^{-1} \smat 0 \\ v_j(s)\srix k_{ij} = \sigma_k^{-1}(2 n_{k, i} x_+(s) -v_i(s)) v_j(s) k_{ij}= \sigma_k^{-1}(2 n_{k, i} x_k(s) +v_i(s)) v_j(s) k_{ij}$ telling us that $R_+(t)= \sigma_k^{-1} \alpha_k(t)$. Defining $\mathcal T_{k, +}(t):=2\mathbb E\int_0^t x_+(s)^\top \Sigma_k^{-1} \big[ f_k(x_k(s))+f_k\big(\small{\smat x_{k-1}(s)\\0\srix}\big)-c_2 x_+(s)\big]ds$ results in
\begin{align*}
\mathbb E\left[x_+(t)^\top \Sigma_k^{-1} x_+(t)\right]&=
\mathbb E\int_0^t x_+(s)^\top\Big((A_k+c_1 I)^\top \Sigma_k^{-1}+ \Sigma_k^{-1} (A_k+c_1 I) + \sum_{i, j=1}^d  N_{k, i}^\top \Sigma_k^{-1} N_{k, j} k_{ij}\Big) x_+(s)ds\\
&\quad+\mathcal T_{k, +}(t)
+\mathbb E  \int_0^t 4 x_+(s)^\top \Sigma_k^{-1} B_k u(s) ds-\sigma_k^{-1} \alpha_k(t)\\
&\quad+c\int_0^t \mathbb E\left[x_+(s)^\top \Sigma_k^{-1} x_+(s) \right] ds. \end{align*}
We exploit the estimate 
\begin{align*}
4\left\| u(s)\right\|_2^2 &\geq \left\|2 u(s)\right\|_2^2 -\left\|B_k^\top \Sigma_k^{-1} x_+(s)- 2u(s)\right\|_2^2 \\ \nonumber
&=-x_+(s)^\top \Sigma_k^{-1} B_kB_k^\top \Sigma_k^{-1} x_+(s)+4 x_+(s)^\top\Sigma_k^{-1} B_k u(s)
                \end{align*}
and insert \eqref{aux_reach_gram} in order to find
\begin{align*}
\mathbb E\left[x_+(t)^\top \Sigma_k^{-1} x_+(t)\right]&\leq 4\left\| u\right\|_{L^2_t}^2
+\mathcal T_{k, +}(t)
-\sigma_k^{-1} \alpha_k(t)
+c\int_0^t \mathbb E\left[x_+(s)^\top \Sigma_k^{-1} x_+(s) \right] ds. \end{align*}
We apply \eqref{gronwall2} providing \begin{align*}
  \int_0^t  \dot \alpha_k(s) \expn^{c(t-s)} ds\leq   \sigma_k\int_0^t  \big(\dot{\mathcal T}_{k, +}(s) +4\mathbb E\left\| u(s)\right\|_{2}^2 \big)\expn^{c(t-s)} ds.
  \end{align*}
Combining this with \eqref{inter_result} leads to \begin{align*}
 \mathbb E \int_0^t  \left\|y_k(s)-y_{k-1}(s)\right\|_{2}^2 \expn^{c(t-s)} ds\leq   \int_0^t\left[\dot{\mathcal T}_{k, -}(s)+  \sigma_k^2\big(\dot{\mathcal T}_{k, +}(s) +4\mathbb E\left\| u(s)\right\|_{2}^2 \big)\right]\expn^{c(t-s)} ds.                                
                                    \end{align*}
The last step is to find different representations for $\mathcal T_{k, -}$ and $\mathcal T_{k, +}$ inserting the definitions of $x_+$ and $x_-$.
We recall that $f_k(x_k):=\tilde f_k(\smat x_k\\ 0_{n-k}\srix)$, $x_k\in \mathbb R^k$ and $0_{n-k}\in \mathbb R^{n-k}$ by \eqref{part_bal}. Since $\tilde f_k$ are the first $k$ entries of the balanced nonlinearity $f_n$, we have \begin{align*}                                                                                                                                                                                                                            &\big(x_k(s)\pm \small{\smat x_{k-1}(s)\\0\srix}\big)^\top D_k \big[ f_k(x_k(s))\pm f_k\big(\small{\smat x_{k-1}(s)\\0\srix}\big)-c_2 \big(x_k(s)\pm \small{\smat x_{k-1}(s)\\0\srix}\big) \big]\\
&= \big(\small{\smat x_{k}(s)\\0_{n-k}\srix}\pm \small{\smat x_{k-1}(s)\\0_{n-k+1}\srix}\big)^\top D_n \big[ f_n(\small{\smat x_{k}(s)\\0_{n-k}\srix})\pm f_n\big(\small{\smat x_{k-1}(s)\\0_{n-k+1}\srix}\big)-c_2 \big(\small{\smat x_{k}(s)\\0_{n-k}\srix}\pm 
\small{\smat x_{k-1}(s)\\0_{n-k+1}\srix} \big) \big],
\end{align*}
 where $D_k\in \{\Sigma_k, \Sigma_k^{-1}\}$. By Proposition \ref{prop_bal} and \eqref{coef_trans}, we know that $\Sigma_n=S^{-\top}QS^{-1}$, $\Sigma_n^{-1}=S^{-\top}P^{-1}S^{-1}$ and $f_n = S f(S^{-1}\cdot)$. Moreover, $S^{-1} \small{\smat x_{k}(s)\\0_{n-k}\srix} = V_k x_k(s)$, since $V_k$ are the first $k$ columns of the inverse $S^{-1}$ of the balancing transformation. Hence,  
 \begin{align*}
  \mathcal T_{k, -}(t)=2 \mathbb E \int_0^t G_{Q}^-\Big(V_k x_k(s), V_{k-1}x_{k-1}(s)\Big) ds,\quad
 \mathcal T_{k, +}(t)=2 \mathbb E \int_0^t G_{P^{-1}}^+\Big(V_k x_k(s), V_{k-1}x_{k-1}(s)\Big) ds 
 \end{align*}
according to the definition of the one-sided Lipschitz gaps in \eqref{PQ_lipschitz_gap}. This concludes the proof using \eqref{remove_HSVs_individually} and setting $t=T$.
\end{proof}  

\section{Numerical experiments}\label{sec_sim}

Below, let $L>0$ defining a ``step size'' parameter $h:=\frac{L}{(n+1)}$. Based on this, we introduce a grid by $\zeta_j = j h$ for $j=0, 1, \dots, n+1$. Now, we mainly focus on an example for \eqref{original_system} that is given by \begin{equation}    \label{finite_difference_model}                                                                                                                                                                                      
\begin{aligned}
dx_1(t) &= \Big[\frac{x_2(t)-2x_1(t)}{h^2} + \frac{u_1(t)}{h^2} +\mathfrak{f}(x_1(t))\Big]dt + \sum_{i=1}^d \mathfrak g_i(\zeta_1) x_1(t-) dM_i(t),\\
dx_j(t) &= \Big[\frac{x_{j+1}(t)-2x_j(t)+x_{j-1}(t)}{h^2} +\mathfrak{f}(x_j(t))\Big]dt + \sum_{i=1}^d \mathfrak g_i(\zeta_j) x_j(t-) dM_i(t), \\
dx_n(t) &= \Big[\frac{-2x_n(t)+x_{n-1}(t)}{h^2} + \frac{u_2(t)}{h^2} +\mathfrak{f}(x_n(t))\Big]dt + \sum_{i=1}^d \mathfrak g_i(\zeta_n) x_n(t-) dM_i(t)                                                                                           \end{aligned}    
\end{equation}
for $j\in\{2, \dots, n-1\}$. We have that $u=\smat u_1\\u_2\srix$ ($m=2$) and $f(x) = \smat {\mathfrak{f}(x_1)} &\dots & {\mathfrak{f}(x_n)} \srix^\top$, where $\mathfrak{f}$ and $\mathfrak{g}_i$ are scalar functions. Formally, \eqref{finite_difference_model} can interpreted as a finite difference discretization of the stochastic reaction diffusion equation 
\begin{equation}\label{stoch_reaction_diff}
\begin{aligned}
 dv_t(\zeta) =\Big[ \frac{\partial^2}{\partial \zeta^2} v_t(\zeta) &+ \mathfrak{f}\big(v_t(\zeta)\big)\Big]  + \sum_{i=1}^d \mathfrak g_i(\zeta) v_{t-}(\zeta) dM_i(t),\quad \zeta\in (0, L), \quad t\in (0, T), \\
 v_0(\zeta)&\equiv 0, \quad
  v_t(0) = u_1(t)\quad\text{and}\quad  v_t(L) = u_2(t),                                                                                                                                 \end{aligned}
\end{equation}
with controlled boundaries and the intuition that $x_j(t)\approx v_t(\zeta_j)$. Let us specify the other parameter and the noise profile. Below, $M$ is a Wiener process                                                                                            in dimension $d=2$ with                                                                                     covariance $K=\smat 1 & -0.5\\-0.5 & 1\srix$ and $n=100$. We study the nonlinearities $\mathfrak f(v) =  (1+a) v^{2} - v^{3}-a v$ with $a=0.1$ and $\mathfrak f(v) = v - v^{3}$, so that $f= f^{(1)}$ or $f= f^{(2)}$ introduced in Example \ref{example1}. The particular noise scaling functions are $\mathfrak g_1(\zeta)= 4 \sin(\zeta)$ and $\mathfrak g_2(\zeta)= 4 \cos(\zeta)$. Moreover, the terminal time is $T=1$ and the quantity of interest shall be the following average: \begin{align}\label{rec_dif_output}
  y(t) = \frac{1}{n} \sum_{j=1}^n x_j(t).                                                                                                                                                                                                                                                                                                                                                                                                                                                                                                                                                                                                                                                                                                                                                                                                                                                                                                                                                                                                                                                                                 \end{align}
For illustration we show two typical paths of \eqref{rec_dif_output} for $f=f^{(1)}, f^{(2)}$ and two different inputs in Figures \ref{output1} and \ref{output2}.
 \begin{figure}[ht]
 \begin{minipage}{0.45\linewidth}
  \hspace{-0.5cm}
 \includegraphics[width=1.0\textwidth,height=6cm]{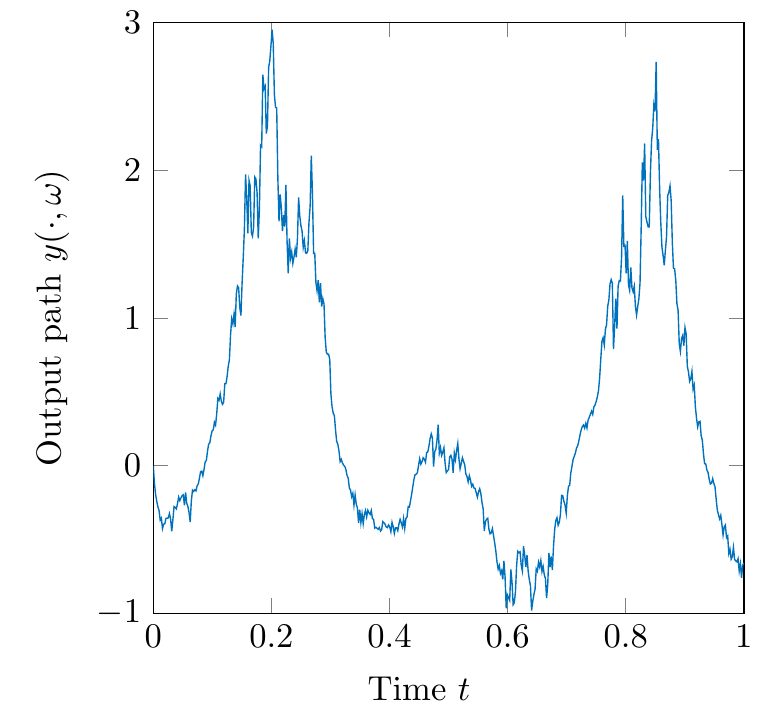}
 \caption{Path of \eqref{rec_dif_output} with $f=f^{(1)}$ and $u=\tilde u$ in \eqref{part_controls}.}\label{output1}
 \end{minipage}\hspace{0.5cm}
 \begin{minipage}{0.45\linewidth}
  \hspace{-0.5cm}
 \includegraphics[width=1.0\textwidth,height=6cm]{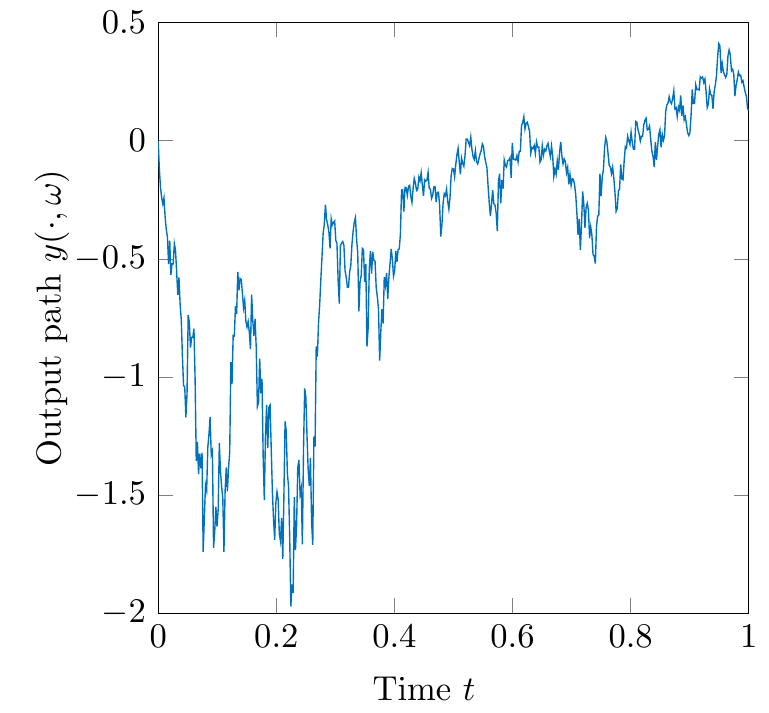}
 \caption{Path of \eqref{rec_dif_output} with $f=f^{(2)}$ and $u=\hat u$ in \eqref{part_controls}.}\label{output2}
 \end{minipage}
 \end{figure}

For $f=f^{(2)}$, we know that \eqref{Xmonoton} holds with $X=I$ and $c_2\geq c_f=1$. Further, we observe that \eqref{shifted_Lyap} is true for $X=I$ and $c_1=c_f=1$. Therefore, the system is globally mean square asymptotically stable according to Theorem \ref{thm_global_stab} and the concept of monotonicity Gramians with $c_1=c_2=1$ is well-defined by Proposition \ref{propzugram2}. We can even guarantee the existence of a one-sided Lipschitz Gramian by Proposition \ref{lip_gram_exist} since the one-sided Lipschitz condition \eqref{one_sided_lip_cond} holds with $c_f=1$ using Example \ref{example5}. The choice of $f=f^{(2)}$ also yields a mean square asymptotically stable system since \eqref{shifted_Lyap} particularly holds for $X=I$ if $c_1=c_f=\frac{(a-1)^2}{4}=0.20250$ is used and since we know, by Example \ref{example1}, that \eqref{Xmonoton} is true setting $X=I$ and $c_2\geq c_f$. Therefore, monotonicity Gramians also exist here for $c_1=c_2=0.20250$. On the other hand, a one-sided Lipschitz Gramian $Q$ exists with $c_1=c_2=\frac{a^2-a+1}{3}=0.30\bar3$ due to Proposition \ref{lip_gram_exist} ($X=I$) exploiting Example \ref{example4}. The same example, however, indicates that $P$ might not be available as a one-sided Lipschitz Gramian.\smallskip

The goal of this section is to construct  average monotonicity Gramians $P$ and $Q$ according to Definition \ref{def_av_mon_Gram} for a large set of controls $\mathcal U$. In detail, we choose the monotonicity/one-sided Lipschitz constant to define $c_1=c_2=1$ for $f=f^{(2)}$ and we set $c_1=c_2=0.30\bar3$ for $f=f^{(1)}$ which is a number dominating the monotonicity constant $0.20250$. Consequently, Theorems \ref{energy_est} and \ref{thm_error_bound} hold for $c=0$. We choose $Q$ to be the solution to the equality in \eqref{newgram2Q} and $P$ the candidate with minimal trace satisfying \eqref{newgram2}. We refer to Section \ref{Sec_comp_gram} for the particular computation strategy. We observe that these $P$ and $Q$ do not satisfy \eqref{PQ_monoton} for all $x\in \mathbb R^n$ but for the essential ones. In fact, we run experiments for a large variety of controls involving increasing, decreasing and (highly) oscillating $u$ as well as a combination of all of them. In all cases, conditions \eqref{P_average_monoton} and \eqref{Q_average_monoton} were fulfilled indicating that these $P$ and $Q$ are average monotonicity Gramians for a large set of controls $\mathcal U\subset L^2_T$. We present the experiments solely for two representatives $\tilde u, \hat u\in \mathcal U$ which are given by \begin{align}\label{part_controls}
\tilde u(t)=\smat -3\cos(20 t)\\ 2\sin(10 t) \srix\quad \text{and}\quad \hat u(t)=\smat -3\expn^{-t}\\ 2\sqrt{t} \srix.\end{align}
These are chosen since they also steer the state $x(t)$ to regions of $\mathbb R^n$, where the monotonicity conditions in \eqref{PQ_monoton} are violated. The constructed monotonicity Gramians have the advantage that the HSVs provide a reliable criterion for the reduction error according to Theorem \ref{energy_est}.  Here, we have $c=0$. We depict these algebraic values for $f=f^{(1)}$ in Figure \ref{hsvplot} and observe a strong decay telling us that we can expect a low approximation error for small $r$. The HSVs for $f=f^{(2)}$ behave very similarly and are therefore omitted. 
\begin{figure}[ht]
\center
\includegraphics[width=7.5cm,height=6.5cm]{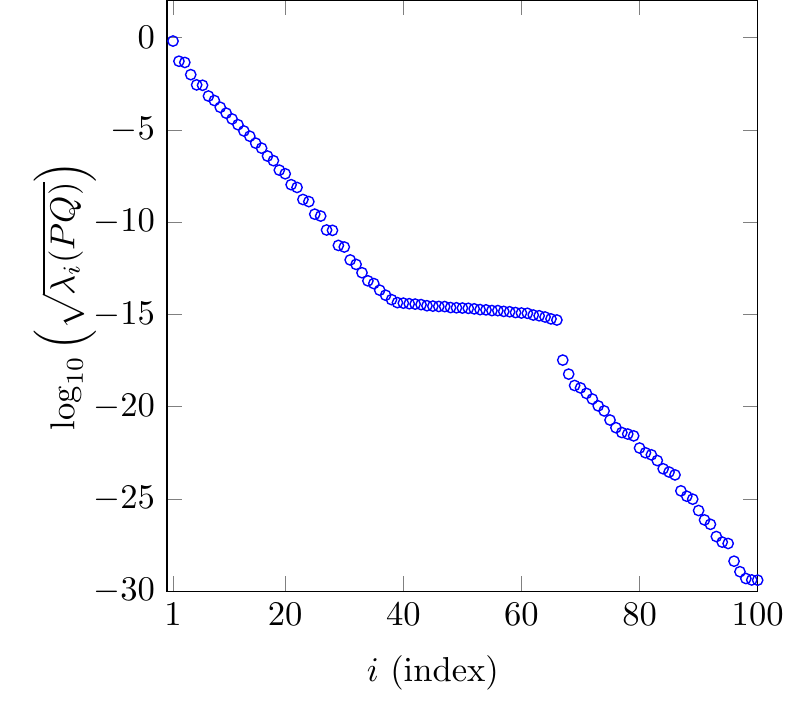}
\caption{Logarithmic HSVs based on monotonicity Gramians for $f=f^{(1)}$ with $c_1=c_2=0.30\bar3$, where $Q$ satisfies the equality in \eqref{newgram2Q} and $P$ is the minimal trace solution of \eqref{newgram2}.}\label{hsvplot}
\end{figure}
As discussed in Remark \ref{rem_error_discussion}, we cannot expect the bound in Corollary \ref{cor_bound} (with $c=0$) to hold if average monotonicity Gramians are used. However, we expect the error to not be far from this bound, since the one-sided Lipschitz gaps $G_{P^{-1}}^+$ and $G_{Q}^-$ in Theorem \ref{thm_error_bound} are expected to be small when they are positive. We compute the output $y_r$ of the reduced order model \eqref{red_system} introduced in Section \ref{sec_BT} for different reduced dimensions $r=3, 6, 10, 20$. The relative output error for $f=f^{(1)}$ can be found in Table \ref{table1} for the controls $\tilde u$ and $\hat u$. We observe a decreasing behaviour for growing $r$ yielding a very high accuracy for $r\geq 6$. Table \ref{table2} shows the bound of Corollary \ref{cor_bound} which generally is no upper bound for the error calculated in Table \ref{table1}, see the case of $r=20$. This is because the one-sided Lipschitz gaps are not always non positive. However, $2\sum_{k=r+1}^n \sigma_k$ is close to the actual error. This is an observation made also in additional simulations that are not presented here. The intuition for $2\sum_{k=r+1}^n \sigma_k$ being an upper bound for dimensions $r=3, 6, 10$ but not for $r=20$ might be the low order of a positive one-sided Lipschitz gap. For that reason, it becomes only visible when $2\sum_{k=r+1}^n \sigma_k$ is very small. 
\begin{table}
\begin{minipage}{.499\linewidth}
\centering\begin{tabular}{|c|c|c|}\hline
& \multicolumn{2}{c}{$\left\|y- y_r\right\|_{L^2_T} /\left\|y\right\|_{L^2_T}$ for $f=f^{(1)}$} \vline \\
\hline 
$r$  & $u=\tilde u$ & $u=\hat u$\\
\hline
\hline
$3$ & $4.4077$e$-02$& $3.8041$e$-02$ \\   
$6$ &$4.0903$e$-03$  &$3.7334$e$-03$ \\
$10$ &$3.1233$e$-04$  & $2.5745$e$-04$ \\
$20$ &$2.7327$e$-07$ & $3.5013$e$-07$\\
\hline
\end{tabular}\caption{Relative output error dimension \\reduction with controls in \eqref{part_controls} and $f=f^{(1)}$.}
\label{table1}
\end{minipage}
\begin{minipage}{.499\linewidth}
\centering\begin{tabular}{|c|c|c|}\hline
& \multicolumn{2}{c}{$2\sum_{k=r+1}^n \sigma_k \left\|u\right\|_{L^2_T} /\left\|y\right\|_{L^2_T}$ for $f=f^{(1)}$} \vline \\
\hline 
$r$  & $u=\tilde u$ & $u=\hat u$\\
\hline
\hline
$3$ &\textcolor{white}{aa} $1.0240$e$-01$\textcolor{white}{aa}& $1.8031$e$-01$ \\   
$6$ &\textcolor{white}{aa} $8.6029$e$-03$ \textcolor{white}{aa}  &$1.5112$e$-02$ \\
$10$ &\textcolor{white}{aa} $4.6198$e$-04$ \textcolor{white}{aa}  & $8.1347$e$-04$ \\
$20$ &\textcolor{white}{aa} $1.3487$e$-07$ \textcolor{white}{aa} & $2.3709$e$-07$\\
\hline
\end{tabular}\caption{Relative error criterion of Corollary \ref{cor_bound} with $c=0$ and $f=f^{(1)}$.}
\label{table2}
\end{minipage}
\end{table}
We repeat the error calculations for $f=f^{(2)}$ and obtain basically the same results, see Tables \ref{table3} and \ref{table4}. This is due to a similar path behaviour of $y$ for both nonlinearities $f^{(1)}$ and $f^{(2)}$.
\begin{table}
\begin{minipage}{.499\linewidth}
\centering\begin{tabular}{|c|c|c|}\hline
& \multicolumn{2}{c}{$\left\|y- y_r\right\|_{L^2_T} /\left\|y\right\|_{L^2_T}$ for $f=f^{(2)}$} \vline \\
\hline 
$r$  & $u=\tilde u$ & $u=\hat u$\\
\hline
\hline
$3$ & $4.3380$e$-02$& $3.5840$e$-02$ \\   
$6$ &$3.7409$e$-03$  &$2.9983$e$-03$ \\
$10$ &$3.1507$e$-04$  & $2.3924$e$-04$ \\
$20$ &$1.8514$e$-07$ & $3.8720$e$-07$\\
\hline
\end{tabular}\caption{Relative output error dimension \\reduction with controls in \eqref{part_controls} and $f=f^{(2)}$.}
\label{table3}
\end{minipage}
\begin{minipage}{.499\linewidth}
\centering\begin{tabular}{|c|c|c|}\hline
& \multicolumn{2}{c}{$2\sum_{k=r+1}^n \sigma_k \left\|u\right\|_{L^2_T} /\left\|y\right\|_{L^2_T}$ for $f=f^{(2)}$} \vline \\
\hline 
$r$  & $u=\tilde u$ & $u=\hat u$\\
\hline
\hline
$3$ & \textcolor{white}{aa}$1.0494$e$-01$\textcolor{white}{aa}& $1.6369$e$-01$ \\   
$6$ &\textcolor{white}{aa}$7.2186$e$-03$  \textcolor{white}{aa}&$1.3624$e$-02$ \\
$10$ &\textcolor{white}{aa}$4.7378$e$-04$\textcolor{white}{aa}  & $7.3326$e$-04$ \\
$20$ &\textcolor{white}{aa}$1.3493$e$-07$\textcolor{white}{aa} & $2.1019$e$-07$\\
\hline
\end{tabular}\caption{Relative error criterion of Corollary \ref{cor_bound} with $c=0$ and $f=f^{(2)}$.}
\label{table4}
\end{minipage}
\end{table}
Let us finally mention that  we conducted the same experiments also when the right Dirichlet boundary condition in \eqref{stoch_reaction_diff} is replaced by the
 Neumann condition  $\frac{\partial}{\partial \zeta} v_t(\zeta)\vert_{\zeta=L} = u_2(t)$     leading to \begin{align*}                                                                                                        
dx_n(t) = \Big[\frac{-x_n(t)+x_{n-1}(t)}{h^2} + \frac{u_2(t)}{h} +\mathfrak{f}(x_n(t))\Big]dt + \sum_{i=1}^d \mathfrak g_i(\zeta_n) x_n(t) dM_i(t)                                                                                                    \end{align*}
 instead of the last line in \eqref{finite_difference_model}. Here, analog results can be seen using the same kind of parameters.

\appendix
\section{Supporting lemmas}
This section contains several useful auxiliary results.
\begin{lem}\label{lemstochdiff}
Suppose that $a, b_1, \ldots, b_d$ are $\mathbb R^n$-valued processes with $a$ being $(\mathcal F_t)$-adapted and almost surely Lebesgue integrable and $b_i$ 
being integrable w.r.t the mean zero square integrable L\'evy process $M=\begin{bmatrix} M_1& \ldots & M_d\end{bmatrix}^\top$ with covariance matrix $K=(k_{ij})$. If $x$ is represented by \begin{align*}
 dx(t)=a(t) dt+ b(t)dM=a(t) dt+ \sum_{i=1}^d b_i(t)dM_i,                                                                                                                                  \end{align*}
where $b=\begin{bmatrix} b_1& \ldots & b_d\end{bmatrix}$. Then, we have \begin{align*}
 \frac{d}{dt}\mathbb E\left[x(t)^\top x(t)\right]=2 \mathbb E\left[x(t)^\top a(t)\right]+ \mathbb E   \left\|b(t)K^{\frac{1}{2}} \right\|_F^2 =2 \mathbb E\left[x(t)^\top a(t)\right] + \sum_{i, j=1}^d \mathbb E\left[b_i(t)^\top b_j(t)\right]k_{ij}.                                                                                                                              \end{align*}
\end{lem}
\begin{proof}
A proof is given in \cite[Lemma 5.2]{redmannspa2}. 
\end{proof}
We introduce two classical versions of Gronwall's lemma below. 
\begin{lem}[Gronwall lemma -- differential form]\label{gron_dif}
Given $T>0$ let $z: [0, T]\rightarrow \mathbb R$ be differentiable functions and $\beta\in \mathbb R$. Given that \begin{align*}
   \dot z(t)\leq  \beta z(t),\quad t\in[0, T],
   \end{align*}
then for all $t\in[0, T]$, it holds that \begin{align*}
 z(t) \leq z(0) \expn^{\beta t}.                                         
\end{align*}
\end{lem}
The corresponding integral version follows next.
\begin{lem}[Gronwall lemma -- integral form]\label{gronwall}
Given $T>0$ let $z, \alpha: [0, T]\rightarrow \mathbb R$ be continuous functions and $\beta\geq 0$. Given that \begin{align*}
    z(t)\leq \alpha(t)+\int_0^t \beta z(s) ds,\quad t\in[0, T],
   \end{align*}
then for all $t\in[0, T]$, it holds that \begin{align}\label{gronwall1}
    z(t)\leq \alpha(t)+\int_0^t \alpha(s)\beta \expn^{\beta (t-s)}ds.
   \end{align}
If $\alpha$ further is absolutely continuous, we have 
\begin{align}\label{gronwall2}
    z(t)\leq \alpha(0) \expn^{\beta t} + \int_0^t \dot\alpha(s)\expn^{\beta (t-s)}ds,
    \end{align}
where $\dot \alpha$ is the derivative of $\alpha$ Lebesgue almost everywhere.
\end{lem}
\begin{proof}
The first part is a very classical result and is not proved here. Given that $\alpha$ is absolutely continuous, we can apply integration by parts yielding \begin{align*}
 \int_0^t \alpha(s)\beta \expn^{\beta (t-s)}ds= -\alpha(s) \expn^{\beta (t-s)}\big\vert_0^t + \int_0^t \dot\alpha(s)\expn^{\beta (t-s)}ds.                                                                                                                                                                              \end{align*}
Hence, we obtain \eqref{gronwall2} from \eqref{gronwall1}.
\end{proof}

\section{Proof of Theorem \ref{thm_global_stab}}\label{appendixB}

We define
\begin{align}\label{defineY}
-Y:=(A+c_1I)^\top X+X (A+c_1I)+\sum_{i, j=1}^d N_i^\top X N_j k_{ij}<0.\end{align} 
We apply Lemma \ref{lemstochdiff} to the uncontrolled process $X^{\frac{1}{2}}x(t)$ and obtain 
\begin{align*}
 \frac{d}{dt}\mathbb E\left[x(t)^\top X x(t)\right]&= 2 \mathbb E\left[x(t)^\top X [Ax(t)+ f(x(t))]\right]  + \sum_{i, j=1}^d \mathbb E\left[x(t)^\top N_i^\top X N_j x(t)\right]k_{ij}\\
 &\leq 2\mathbb E\left[x(t)^\top X [Ax(t)+ c_2 I x(t)]\right]  + \sum_{i, j=1}^d \mathbb E\left[x(t)^\top N_i^\top X N_j x(t)\right]k_{ij} \\ &= \mathbb E\bigg[x(t)^\top\Big((A+c_1I)^\top X+ X (A + c_1 I ) + \sum_{i, j=1}^d  N_i^\top X N_j k_{ij}\Big) x(t)\bigg]\\&\quad + 2(c_2-c_1) \mathbb E\left[x(t)^\top X x(t)\right]
 \\
 &=2(c_2-c_1) \mathbb E\left[x(t)^\top X x(t)\right]-\mathbb E\left[x(t)^\top Y x(t)\right]                                                                                                                             \end{align*}
exploiting inequality \eqref{Xmonoton} and inserting \eqref{defineY}.
We define $\underline k$ and  $\overline k$ to be the smallest the largest eigenvalue of $X$, respectively, yielding  $\underline k I\leq X \leq \overline k I$. 
With the smallest eigenvalue $k_Y$ of $Y$ giving $-Y\leq -k_Y I$, we obtain $-\mathbb E\left[x(t)^\top Y x(t)\right] \leq -k_Y \mathbb E\left[x(t)^\top x(t) \right]\leq -\frac{k_Y}{\overline k}\mathbb E\left[x(t)^\top X x(t)\right]$. Setting $\beta:= \frac{k_Y}{\overline k}$, we hence find\begin{align*}
\frac{d}{dt}\mathbb E\left[x(t)^\top X x(t)\right]\leq  (2(c_2-c_1) -\beta)\mathbb E\left[x(t)^\top X x(t)\right].\end{align*}
By the differential version of Gronwall's inequality in Lemma \ref{gron_dif}, we have \begin{align*}
\mathbb E\left[x(t)^\top x(t) \right]&\leq\frac{1}{\underline k}\mathbb 
E\left[x^\top(t) X x(t) \right]\leq \frac{1}{\underline k} x^\top_0 X x_0 
\exp\left\{(2(c_2-c_1) -\beta)t\right\}\\
&\leq \frac{\overline k}{\underline k} x^\top_0 x_0 
\exp\left\{(2(c_2-c_1) -\beta)t\right\}
\end{align*}
concluding the proof.\hfill\qedsymbol

\section*{Acknowledgments}
 MR is supported by the DFG via the individual grant ``Low-order approximations for large-scale problems arising in the context of high-dimensional
PDEs and spatially discretized SPDEs''-- project number 499366908.

\bibliographystyle{abbrv}

\begin{thebibliography}{10}

\bibitem{exisence_uniqueness_levy}
S.~Albeverio, Z.~Brzeźniak, and J.-L. Wu.
\newblock {Existence of global solutions and invariant measures for stochastic
  differential equations driven by Poisson type noise with non-Lipschitz
  coefficients}.
\newblock {\em Journal of Mathematical Analysis and Applications},
  371(1):309--322, 2010.

\bibitem{beckerhartmann}
S.~Becker and C.~Hartmann.
\newblock {Infinite-dimensional bilinear and stochastic balanced truncation
  with error bounds}.
\newblock {\em Math. Control. Signals, Syst.}, 31:1--37, 2019.

\bibitem{breiten_benner2}
P.~Benner and T.~Breiten.
\newblock {Two-Sided Projection Methods for Nonlinear Model Order Reduction}.
\newblock {\em SIAM J. Sci. Comput.}, 37(2), 2015.

\bibitem{bennerdammcruz}
P.~Benner, T.~Damm, and Y.~R. Rodriguez~Cruz.
\newblock {Dual pairs of generalized Lyapunov inequalities and balanced
  truncation of stochastic linear systems}.
\newblock {\em IEEE Trans. Autom. Contr.}, 62(2):782--791, 2017.

\bibitem{morBenG17}
P.~Benner and P.~Goyal.
\newblock {Balanced truncation model order reduction for quadratic-bilinear
  systems}.
\newblock Technical Report 1705.00160, arXiv, 2017.

\bibitem{nonlinear_irka}
P.~Benner, P.~Goyal, and S.~Gugercin.
\newblock {$\mathcal H_2$-Quasi-Optimal Model Order Reduction for
  Quadratic-Bilinear Control Systems}.
\newblock {\em SIAM J. Matrix Anal. Appl.}, 39(2), 2018.

\bibitem{redmannbenner}
P.~{Benner} and M.~{Redmann}.
\newblock {Model Reduction for Stochastic Systems.}
\newblock {\em {Stoch PDE: Anal Comp}}, 3(3):291--338, 2015.

\bibitem{react_dif_daprato}
G.~Da~Prato.
\newblock {\em {Kolmogorov Equations for Stochastic PDEs}}.
\newblock Number VII. Birkhäuser Basel, 2004.

\bibitem{damm}
T.~Damm.
\newblock {\em {Rational Matrix Equations in Stochastic Control.}}
\newblock {Lecture Notes in Control and Information Sciences 297. Berlin:
  Springer}, 2004.

\bibitem{BT_bound_enns}
D.~F. Enns.
\newblock {Model reduction with balanced realizations: An error bound and a
  frequency weighted generalization}.
\newblock {\em In Proc. 23rd IEEE Conf. Decision and Control}, 1984.

\bibitem{BT_bound_glover}
K.~Glover.
\newblock {All optimal Hankel-norm approximations of linear multivariable
  systems and their $L^\infty$-error bounds}.
\newblock {\em Int. J. Control}, 39(6):1115--1193, 1984.

\bibitem{gosea_antoulas}
I.~V. Gosea and A.~C. Antoulas.
\newblock {Data-driven model order reduction of quadratic-bilinear systems}.
\newblock {\em Numerical Linear Algebra with Applications}, 25(6), 2018.

\bibitem{irka}
S.~Gugercin, A.~C. Antoulas, and C.~Beattie.
\newblock {$\mathcal H_2$ Model Reduction for Large-Scale Linear Dynamical
  Systems}.
\newblock {\em SIAM J. Matrix Anal. Appl.}, 30(2):609--638, 2008.

\bibitem{gyongy1}
I.~Gyöngy.
\newblock {Lattice Approximations for Stochastic Quasi-Linear Parabolic Partial
  Differential Equations Driven by Space-Time White Noise I}.
\newblock {\em Potential Analysis}, 9:1--25, 1998.

\bibitem{gyongy2}
I.~Gyöngy.
\newblock {Lattice Approximations for Stochastic Quasi-Linear Parabolic Partial
  Differential Equations Driven by Space-Time White Noise II}.
\newblock {\em Potential Analysis}, 11:1--37, 1999.

\bibitem{gyongy0}
I.~Gyöngy.
\newblock {On stochastic finite difference schemes}.
\newblock {\em {Stoch PDE: Anal Comp}}, 2:539--583, 2014.

\bibitem{staboriginal}
R.~Z. Khasminskii.
\newblock {\em Stochastic stability of differential equations}, volume~66 of
  {\em Stochastic Modelling and Applied Probability}.
\newblock Springer, Heidelberg, second edition, 2012.

\bibitem{pod_kramer}
B.~Kramer and K.~E. Willcox.
\newblock {Nonlinear Model Order Reduction via Lifting Transformations and
  Proper Orthogonal Decomposition}.
\newblock {\em AIAA Journal}, 57(6):2297--2307, 2019.

\bibitem{Kramer2022}
B.~Kramer and K.~E. Willcox.
\newblock {\em {Balanced Truncation Model Reduction for Lifted Nonlinear
  Systems}}, pages 157--174.
\newblock Springer International Publishing, Cham, 2022.

\bibitem{pod}
K.~Kunisch and S.~Volkwein.
\newblock {Galerkin proper orthogonal decomposition methods for parabolic
  problems}.
\newblock {\em Numer. Math.}, 90(1):117--148, 2001.

\bibitem{kuehn_neamtu}
C.~Kühn and A.~Neamţu.
\newblock {Dynamics of Stochastic Reaction-Diffusion Equations}.
\newblock In W.~Grecksch and H.~Lisei, editors, {\em {Infinite Dimensional and
  Finite Dimensional Stochastic Equations and Applications in Physics}},
  chapter~1, pages 1--60. World Scientific Publishing, 2020.

\bibitem{yalmip}
J.~L{\"{o}}fberg.
\newblock {YALMIP: A Toolbox for Modeling and Optimization in MATLAB}.
\newblock In {\em In Proceedings of the CACSD Conference}, Taipei, Taiwan,
  2004.

\bibitem{mao}
X.~Mao.
\newblock {\em Stochastic Differential Equations and Applications (Second
  Edition)}.
\newblock Woodhead Publishing, 2007.

\bibitem{marinelli_roeckner}
C.~Marinelli and M.~Röckner.
\newblock {On uniqueness of mild solutions for dissipative stochastic evolution
  equations}.
\newblock {\em {Infinite Dimensional Analysis, Quantum Probability and Related
  Topics}}, 13(3):363--376, 2010.

\bibitem{moo1981}
B.~Moore.
\newblock Principal component analysis in linear systems: Controllability,
  observability, and model reduction.
\newblock {\em IEEE transactions on automatic control}, 26(1):17--32, 1981.

\bibitem{mosek}
{MOSEK ApS}.
\newblock {\em {The MOSEK optimization toolbox for MATLAB manual. Version
  10.0.}}, 2022.

\bibitem{zabczyk}
S.~Peszat and J.~Zabczyk.
\newblock {\em {Stochastic Partial Differential Equations with L\'evy Noise. An
  evolution equation approach.}}
\newblock {Encyclopedia of Mathematics and Its Applications 113. Cambridge
  University Press}, 2007.

\bibitem{qian_data}
E.~Qian, B.~Kramer, B.~Peherstorfer, and K.~E. Wilcox.
\newblock {Lift \& Learn: Physics-informed machine learning for large-scale
  nonlinear dynamical systems}.
\newblock {\em Physica D: Nonlinear Phenomena}, 2020.

\bibitem{redstochbil}
M.~Redmann.
\newblock {Energy estimates and model order reduction for stochastic bilinear
  systems}.
\newblock {\em Int. J. Control}, 93(8):1954--1963, 2018.

\bibitem{redmannspa2}
M.~Redmann.
\newblock {Type II singular perturbation approximation for linear systems with
  Lévy noise}.
\newblock {\em SIAM J. Control Optim.}, 56(3):2120--2158, 2018.

\bibitem{mliopt}
M.~Redmann and M.~A. Freitag.
\newblock {Optimization based model order reduction for stochastic systems}.
\newblock {\em Appl. Math. Comput.}, Volume 398, 2021.

\bibitem{Scherpen}
J.~M.~A. Scherpen.
\newblock {Balancing for nonlinear systems}.
\newblock {\em Syst. Control. Lett.}, 21:143--153, 1993.

\bibitem{shardlow}
T.~Shardlow.
\newblock {Numerical methods for stochastic parabolic PDEs}.
\newblock {\em Numerical Functional Analysis and Optimization},
  20(1-2):121--145, 1999.

\bibitem{pod_sde}
T.~M. Tyranowski.
\newblock {Data-driven structure-preserving model reduction for stochastic
  Hamiltonian systems}.
\newblock {\em arXiv preprint:2201.13391}, 2022.

\end{thebibliography}

\end{document}